\documentclass[11pt,a4paper,twoside]{article}
\usepackage{latexsym,amsfonts,amsmath, amsthm, amssymb}
\usepackage{fullpage}
\usepackage{xcolor,bm}
\usepackage[utf8x]{inputenc}
\usepackage{array}
\usepackage{mathrsfs}
\usepackage{comment}

\usepackage{fourier}

\usepackage{graphicx}       % standard LaTeX graphics tool
\newcommand{\R}{\mathbb R}
\newcommand{\N}{\mathbb N}

\newcommand{\E}{\mathbb E}

\newcommand{\Sc}{\mathcal{S}}

\newcommand{\SSS}{\ensuremath{{\mathbb S}}}

\DeclareMathOperator{\conv}{conv}

\DeclareMathOperator{\codim}{codim}
\DeclareMathOperator{\linspan}{span} %JAN

\newtheorem{thm}{Theorem}[section]

\newtheorem{lemma}[thm]{Lemma}

\newtheorem{proposition}[thm]{Proposition}

\newtheorem{thmalpha}{Theorem}

{
\theoremstyle{definition}

\newtheorem{rmk}[thm]{Remark}

}
\allowdisplaybreaks
\begin{document}

\title{\bf Gelfand numbers of embeddings of Schatten classes}

\medskip

\author{ Aicke Hinrichs \and Joscha Prochno \and Jan Vyb\'iral}
%\address[Aicke Hinrichs]{Institut f\"ur Analysis\\
%Johannes Kepler Universit\"at Linz\\
%Altenbergerstrasse 69\\
%4040 Linz\\
%Austria}
%\email{aicke.hinrichs@jku.at}
%
%\address[Joscha Prochno]{School of Mathematics \& Physical Sciences\\
%University of Hull\\
%Cottingham Road\\
%Hull HU6 7RX \\
%United Kingdom}
%\email{j.prochno@hull.ac.uk}
%
%\address[Jan Vyb\'iral]{Department of Mathematics\\
%Czech Technical University\\
%Trojanova 13\\
%12000 Praha\\
%Czech Republic}
%\email{jan.vybiral@fjfi.cvut.cz}
%\thanks{~}

%% NB There should be only one primary classification, and zero or
%more secondary classifications.

%\date{}

\maketitle

\begin{abstract}
\small
%\footnotesize
Let $0<p,q\leq \infty$ and denote by $\Sc_p^N$ and $\Sc_q^N$ the corresponding Schatten classes of real $N\times N$ matrices. We study the Gelfand numbers of natural identities $\Sc_p^N\hookrightarrow \Sc_q^N$ between Schatten classes and prove asymptotically sharp bounds up to constants only depending on $p$ and $q$. This extends classical results for finite-dimensional $\ell_p$ sequence spaces by E. Gluskin to the non-commutative setting and complements bounds previously obtained by B. Carl and A. Defant, A. Hinrichs and  C. Michels, and J. Ch\'avez-Dom\'inguez and D. Kutzarova.
\medspace
\end{abstract}

%\tableofcontents

% % % % % % % % % % % % % % % % % % % % % % % % % % % % % % % % %
\section{Introduction and main result}
% % % % % % % % % % % % % % % % % % % % % % % % % % % % % % % % %

The Schatten $p$-class $\Sc_p$ ($0<p\leq \infty$) is the collection of all compact operators between Hilbert spaces for which the sequence of their singular values belongs to the sequence space $\ell_p$, thus including the important cases of the trace class operators ($p=1$) and Hilbert-Schmidt operators ($p=2$). Having been introduced by R. Schatten in \cite[Chapter 6]{Schatten1960}, inspired by the works of John von Neumann, they are among the most prominent unitary operator ideals studied in functional analysis today. Schatten classes provide the mathematical framework to modern applied mathematics around low-rank matrix recovery and completion (see, e.g., \cite{CR2009, CDK2015, FR2013, KS2018, RT2011} and the references cited therein) and are fundamental in quantum information theory, in particular in connection to counterexamples to Hasting's additivity conjecture (see, e.g., \cite{AS2017, ASW2010, ASW2011}). The Schatten class $\Sc_p$ may be considered a non-commutative version of the classical $\ell_p$ sequence space and both share various structural characteristics. They are lexicographically ordered, uniformly convex whenever $1<p<\infty$, and satisfy a (trace) duality relation together with a corresponding H\"older inequality. However, while there are several similarities on different levels, there are also many differences in their analytic, geometric, and probabilistic behavior. While the matrix spaces are easier to handle in certain situations, there are other situations in which arguments are considerably more delicate and complicated.

From both the local and the global point of view the study of Schatten classes has a long tradition in geometric functional analysis and their structure has been investigated intensively in the past 50 years. Gordon and Lewis proved that $\Sc_p$ for $p\neq 2$ fails to have local unconditional structure and therefore does not have an unconditional basis \cite{GL1974}. This answered a question of Kwapie\'n and Pe{\l}czy\'nski who had previously shown in \cite{KP1970} that the Schatten trace class $\Sc_1$ (naturally identified with $\ell_2\otimes_{\pi}\ell_2$) as well as $\Sc_\infty$ are not isomorphic to subspaces with an unconditional basis. 
In 1974, Tomczak-Jaegermann succeeded in \cite{TJ1974} to prove that $\Sc_1$ has Rademacher cotype $2$, and K\"onig, Meyer, and Pajor obtained in \cite{KMP1998} that the isotropic constants of the unit balls in $\Sc_p^N$ are bounded above by absolute constants for all $1\leq p\leq \infty$. The concentration of mass properties of unit balls in the Schatten $p$-classes were studied by Gu\'edon and Paouris in \cite{GP2007}, Radke and Vritsiou were able to prove the thin-shell conjecture for $\Sc_\infty^N$ \cite{RV2016}, and Hinrichs, Prochno, and Vyb\'iral determined the asymptotic behavior of entropy numbers for natural identities $\Sc_p^N\hookrightarrow\Sc_q^N$ up to absolute constants for all $0<p,q\leq \infty$ \cite{HPV2017}. In a series of papers, Kabluchko, Prochno, and Th\"ale computed the precise asymptotic volume and the volume ratio for Schatten $p$-classes for $0<p\leq \infty$ \cite{KPT2018a}, proved a Schechtman-Schmuckenschl\"ager type result for the volume of intersections of unit balls \cite{KPT2018b}, and obtained Sanov-type large deviations for the empirical spectral measures of random matrices in Schatten unit balls \cite{KPT2018c}. Recent years have again seen an increased interest in Schatten trace classes in parts because of their role in low-rank matrix recovery, the non-commutative analogue to the classical compressed sensing approach. 
We refer the reader to \cite{CDK2015, FR2013} and the references cited therein for more information.

The study of compact linear operators between Banach spaces, i.e., those for which the image under the operator of any bounded subset of the domain is a relatively compact subset (has compact closure) of the codomain, is one of the central aspects of functional analysis and Banach space theory in particular. The interest originates in the theory of integral equations, because integral operators are typical examples of compact operators. One way to quantify the degree of compactness of an operator is via its sequence of Gelfand numbers. Those are an important concept in approximation and complexity theory as well as in Banach space geometry. Given (quasi-)Banach spaces $X,Y$ and a bounded linear operator $T\in\mathscr L(X,Y)$, the $n$-th Gelfand number of $T$ is defined by
\[
c_n(T):= \inf\big\{ \|T|_F\| \,:\, F\subset X,\,\codim F < n \big\}.
\]
On the application side, Gelfand numbers (of canonical embeddings) naturally appear when considering the problem of optimal recovery of an element $f\in X$ from few arbitrary linear samples, where the recovery error is measured in the norm of the codomain space $Y$, which substantiates their role in the flourishing fields of information-based complexity (see, e.g., \cite{N2017, NW2008}) and approximation theory \cite{CS1990, CT2019, ELN2009, P1985}.

The systematic study of Gelfand numbers (and widths) for natural  embeddings between the finite-dimensional classical $\ell_p$ sequence spaces has a long tradition and can be traced back as far as the work of Stechkin~\cite{Stech1954}. After contributions of Stesin \cite{Ste1975}, Ismagilov \cite{I1974}, Kashin \cite{Ka1974, Ka1977} and others, it were eventually the works of Gluskin \cite{G81,G83} and Garnaev and Gluskin \cite{GG1984} that settled the question about the order of Gelfand numbers (and widths) completely. While lower bounds had been obtained before, the eventual breakthrough regarding sharp asymptotic upper bounds was made using random approximation, a groundbreaking and powerful method having its origin in the work of Kashin \cite{Ka1977}. There are a number of extensions and generalizations together with fascinating applications, for instance, estimates for Gelfand numbers of operators with values in a Hilbert space (relating Gelfand numbers and certain Gaussian/Rademacher averages) \cite{CP1988}, sharp asymptotic bounds for Gelfand numbers (widths) of natural embeddings in the quasi-Banach space regime $0<p\leq 1$ and $p<q\leq 2$ with applications in compressive sensing \cite{FPRU2010}, or sharp asymptotic bounds for Gelfand numbers of mixed-(quasi-)norm canonical embeddings of $\ell_p^N(\ell_q^M)$ into $\ell_r^N(\ell_s^M)$ with applications to optimality assertions for the recovery of block-sparse and sparse-in-level vectors \cite{DU2017}, just to name a few. 

In the non-commutative setting of Schatten classes much less is known about the order of Gel\-fand numbers of natural embeddings, but applications demonstrate the importance of understanding their behavior in form of quantitative bounds on their decay rate. Let us elaborate on what is known so far.
It was proved by Carl and Defant in \cite{carldefant1997} that for $1\leq n \leq N^2$ and $1\le p\le 2$,
\begin{equation}\label{eq:CD}
c_n\big(\Sc_p^N\hookrightarrow \Sc_2^N\big) \asymp_p \min\bigg\{ 1,\frac{N^{3/2-1/p}}{n^{1/2}}\biggr\},
\end{equation}
where $\asymp_p$ denotes equivalence up to constants depending only on $p$ and for $0<p,q\leq \infty$, $\Sc_p^N \hookrightarrow \Sc_q^N$ denotes the
natural identity map from $\Sc_p^N$ to $\Sc_q^N$.
In the same paper \cite[Remark 2, page 251]{carldefant1997} it is proved that for $1\leq n \leq N^2$ and $2\le q \le \infty$, 
\begin{equation}\label{eq:CD2}
c_n\big(\Sc_2^N\hookrightarrow \Sc_q^N\big) \asymp_q \max\bigg\{ N^{1/q-1/2},\Big( \frac{N^2-n+1}{N^2} \Big)^{1/2}\biggr\},
\end{equation}
where we implicitly used the equality of approximation and Gelfand numbers for operators defined on a Hilbert space \cite{Pietsch2} (see also \cite[Lemma 1.2]{GKS1987} for their equivalence in the case of type $2$ spaces).
%\textcolor{red}{JA: check the dependence of the constants!}
Two decades later, following ideas of Foucart, Pajor, Rauhut, and Ullrich from \cite{FPRU2010}, Ch\'avez-Dom\'in\-guez and Kutzarova obtained an asymptotic formula in the quasi-Banach space setting \cite{CDK2015}, proving that
\begin{equation}\label{eq:cdk}
c_n\big(\Sc_p^N\hookrightarrow \Sc_q^N\big) \asymp_{p,q} \min\bigg\{ 1,\frac{N}{n} \bigg\}^{1/p-1/q},
\end{equation}
whenever $0<p\leq 1$ and $p<q\leq 2$, where the lower bound carries over to $q>2$. Here $\asymp_{p,q}$ refers to equivalence up to constants depending only on $p$ and $q$.  As we pointed out before, their result has very nice applications in the theory of low-rank matrix recovery and we refer the reader to \cite{CDK2015} and the references cited therein for a detailed discussion. The result of Ch\'avez-Dom\'inguez and Kutzarova, and in particular the one of Carl and Defant, is complemented by asymptotic bounds obtained by Hinrichs and Michels \cite{HM2005}. They proved in \cite[Proposition 4.1]{HM2005} and \cite[Example 4.7]{HM2005} that for $1\leq q \leq 2$ and $q<p\leq \infty$,
\begin{equation} % referenced as (4) in diagrams
c_n\big(\Sc_p^N\hookrightarrow \Sc_q^N\big) \asymp_{p,q} 
\begin{cases}
\bigg(\frac{N^2-n+1}{N}\bigg)^{1/q-1/p} 
 &:\, 
 1\leq n \leq N^2-N+1 \\
1 
& :\,   
N^2-N+1 < n \leq N^2.
\end{cases}
\end{equation}
Moreover, in \cite[Example 4.14]{HM2005} they obtained an asymptotic lower bound, showing that whenever $2<p<q<\infty$, 
\begin{equation} % referenced as (5) in diagrams
c_n\big(\Sc_p^N\hookrightarrow \Sc_q^N\big) \gtrsim_{p,q} 
\begin{cases}
\sqrt{\frac{N^2-n+1}{N^2}}^{\,\frac{1/p-1/q}{1/2-1/q}} & :\, 1\leq n \leq N^2-N^{2/q+1}+1 \\
N^{1/q-1/p}  & :\, N^2-N^{2/q+1}+1 < n \leq N^2,
\end{cases}
\end{equation}
where $\gtrsim_{p,q}$ refers to the lower bound holding up to constants depending on $p$ and $q$.
The left diagram in Figure \ref{fig:before and after} summarizes the known results.

In this paper, we complement those results for Gelfand numbers of natural embeddings between Schatten classes and provide asymptotically sharp bounds for almost all missing regimes. 
In the case $0<q\le p \le \infty$ we extend the known bounds for $1\le q \le \min(p,2)$ to the complete range and, therefore, settle this case completely, including the case of quasi-Banach spaces.
When $p<q$ we complement the existing bounds from \cite{carldefant1997,CDK2015,HM2005} in the following way. We settle the case $1<p<q\le 2$ completely by providing new sharp asymptotic lower and upper bounds. 
When $1\le p < 2 < q \le \infty$ we provide sharp bounds for all $n$ with $1\le n\le c_{p,q} N^2$ and $ N^2 - N^{1+2/q}+1 \le n \le  N^2$, where $c_{p,q}\in(0,1)$ depends only on $p,q$. In the remaining strip  $c_{p,q} N^2<n < N^2 - N^{1+2/q}+1$, we provide an upper bound which is in fact sharp for $p=2$. 
In the case $2 < p < q \le \infty$, the known lower bound from \cite{HM2005} is already sharp for  $1\le n\le c_{p,q} N^2$, since it is up to constants equal to the norm.
We also show that the lower bound from \cite{HM2005} is sharp  whenever  $ N^2 - N^{1+2/q}+1 \le n \le  N^2$.
In the intermediate regime $c_{p,q} N^2<n < N^2 - N^{1+2/q}+1$, we provide a new upper bound leaving a gap compared to the lower bound from \cite{HM2005}.
We also show that the bounds for the range $c_{p,q} N^2<n < N^2 $ carry over from  $1\le p < 2 < q \le \infty$ to the quasi-Banach case $0<p<1 <2 \le q \le \infty$ and are again asymptotically sharp when $N^2 - N^{1+2/q}+1 \leq n \leq N^2$.
In this quasi-Banach case, we also provide an upper and lower bound in the range $1\leq n \leq c_{p,q}N^2$ matching the one from \cite{CDK2015} for $q=2$.
The right diagram in Figure \ref{fig:before and after} summarizes both known and new results.

The main result of this paper is thus the following, where  it will become clear from the proofs in Sections 
\ref{The case q<p}  and \ref{The case p<q} whether the respective constants indeed depend on the parameters $p$ and/or $q$ or not.

\begin{figure}\label{fig:before and after}
   \begin{minipage}{.497\linewidth} % [b] => Ausrichtung an \caption
      \includegraphics[width=\linewidth]{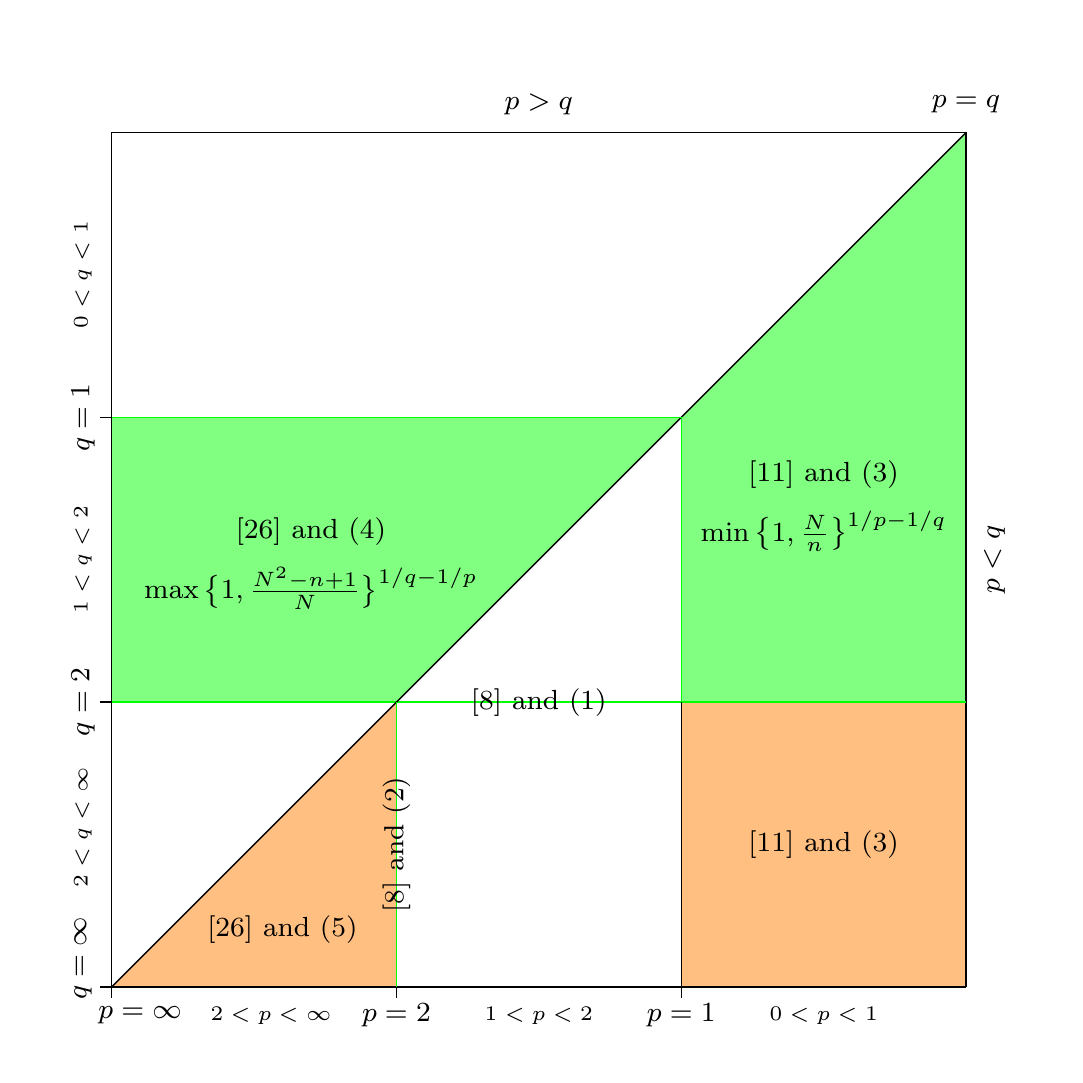}
      %\caption{Before}
   \end{minipage}
   %\hspace{.1\linewidth}% Abstand zwischen Bildern
   \begin{minipage}{.497\linewidth} % [b] => Ausrichtung an \caption
      \includegraphics[width=\linewidth]{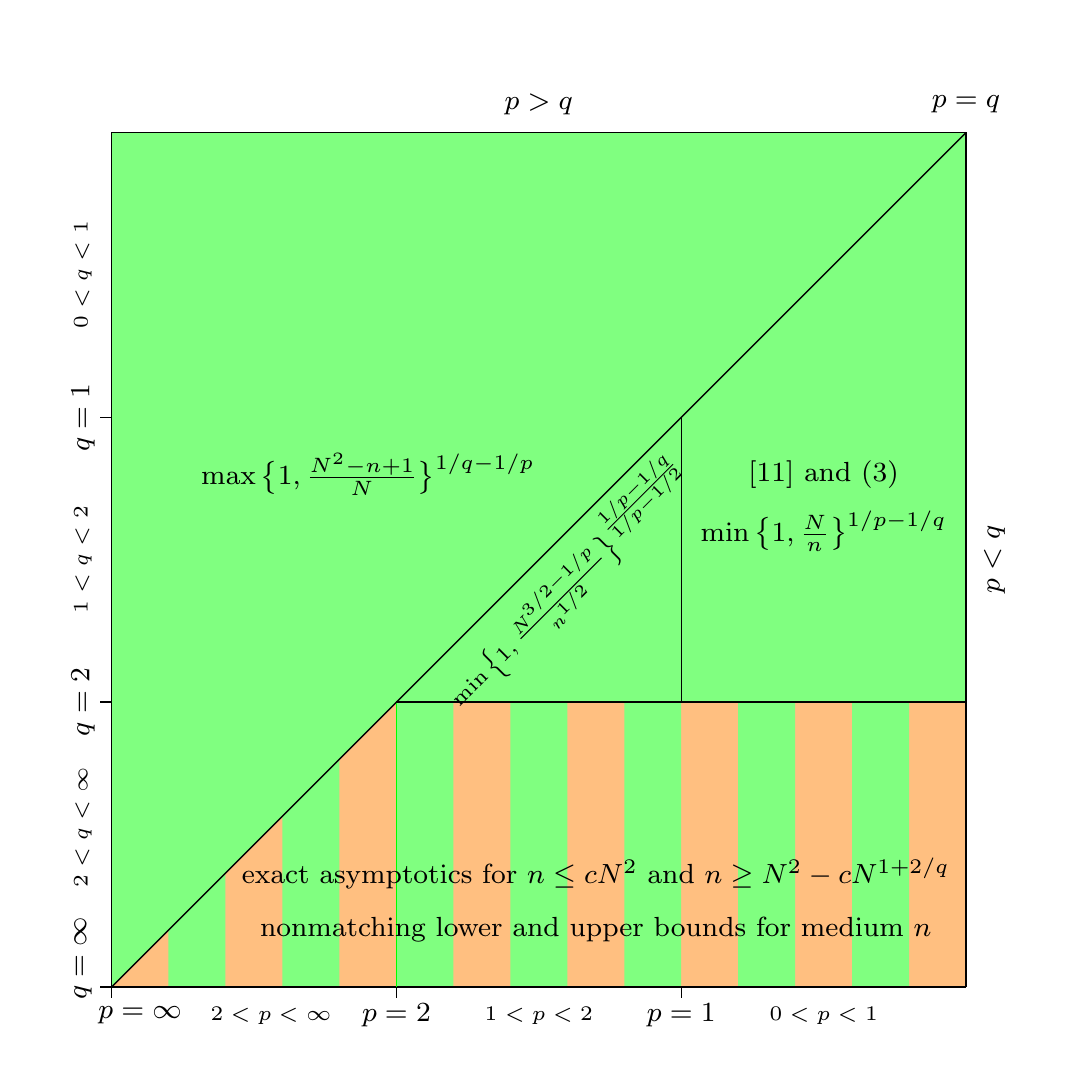}
      %\caption{After}
   \end{minipage}
   \caption{Summary of results known before (left) and now from Theorems \ref{thm:main sharp bounds} and \ref{thm:main non-sharp bounds}  (right). Areas with exact asymptotics are green, nonmatching nontrivial bounds are pink.}
\end{figure}

\begin{thmalpha}\label{thm:main sharp bounds}
Let $0<p,q\,\le\,\infty$ and assume that $n,N\in\N$ with $1\leq n \leq N^2$. Then
\[
c_n\big(\Sc_p^N\hookrightarrow \Sc_q^N\big)  
\asymp_{p,q}
\begin{cases}
\max\left\{1,\frac{N^2 - n +1}{N} \right\} ^{1/q-1/p} & :\, \ 0<q \le p \le \infty \\
\min\bigg\{ 1,\frac{N}{n} \bigg\}^{1/p-1/q}
& :\, \ 0<p\leq 1 \text{ and } p<q\leq 2\\
\min\Biggl\{1,\frac{N^{3/2-1/p}}{n^{1/2}}\Biggr\}^{\frac{1/p-1/q}{1/p-1/2}}
& :\, \ 1\leq p \leq q\leq 2\\
\min\Big\{1,\frac{N^{3/2-1/p}}{n^{1/2}}\Big\} & :\, \ 1 < p \le 2 \le q\le \infty \text{ and } 1 \leq n \leq c_{p,q} N^2\\
1 & :\, \ 2 \le p \le q \le \infty \text{ and } 1 \leq n \leq c_{p,q} N^2\\
N^{1/q-1/p} & :\, \ 0 < p \le q \le \infty \text{ and }  N^2 - c N^{1+2/q}+1 \leq n \leq  N^2. 
\end{cases}
\]
Here $c_{p,q} \in (0,1)$ is a constant depending on $p$ and $q$ and $c\in(0,\infty)$ is an absolute constant.
\end{thmalpha}

As was explained before, the previous theorem provides asymptotically sharp bounds in almost all cases. In the Banach space setting, only for $2 < p  \le q \le \infty$ and the intermediate range $c_{p,q} N^2\le n \le N^2 - c N^{1+2/q}+1$ there remains some gap. In the quasi-Banach case $0<p\leq 1$ and $q\geq 2$ there remains some gap in upper and lower bounds in the ranges of small and intermediate codimensions. In the following theorem, we collect upper and lower bounds that can be established in those cases.

\begin{thmalpha}\label{thm:main non-sharp bounds}
Let $0<p,q\leq\infty$ and assume that $n,N\in\N$ with $1\leq n \leq N^2$. Then the following estimates hold:
\begin{enumerate}
\item If\, $0<p\leq 1$, $2\leq q\leq \infty$, and $1\leq n \leq c_{p,q}N^2$, then
\[
\min\bigg\{ 1,\frac{N}{n} \bigg\}^{1/p-1/q} \lesssim_{p,q} c_n\big(\Sc_p^N\hookrightarrow \Sc_q^N\big)  
\lesssim_{p}\min\Big\{1,\frac{N}{n} \Big\}^{1/p-1/2},
\]
which is sharp up to constants for $q=2$.
\item If\, $0<p \le 1,2  \le q \le \infty$ and $ c_{p,q} N^2 \le n \le N^2 - c N^{1+2/q}+1$, then
\[
\min\bigg\{ 1,\frac{N}{n} \bigg\}^{1/p-1/q} \lesssim_{p,q} c_n\big(\Sc_p^N\hookrightarrow \Sc_q^N\big)  
\lesssim_{q} N^{-1/p-1/2} (N^2 - n +1)^{1/2}.
\] 
Note that the upper bound remains valid as long as $0<p \le 2  \le q \le \infty$.
%and $ c_{p,q} N^2 \le n \le N^2 - c N^{1+2/q}+1$, then we still have an upper bound
%\[
% c_n\big(\Sc_p^N\hookrightarrow \Sc_q^N\big)  
%\lesssim_{q} N^{-1/p-1/2} (N^2 - n +1)^{1/2}.
%\] 

\item If\, $2 \le p  \le q \le \infty$ and $N^2 - c_q^{-2} \, N^{1+2/p} +1  \le n \le N^2 - c N^{1+2/q}+1$, then 
\[
\sqrt{\frac{N^2-n+1}{N^2}}^{\,\frac{1/p-1/q}{1/2-1/q}}
 \lesssim_{p,q}
 c_n\big(\Sc_p^N\hookrightarrow \Sc_q^N\big) 
 \lesssim_{q} N^{1/2-1/p} \sqrt{\frac{N^2-n+1}{N^2}},
\] 
which is sharp up to constants for $p=2$. Note that when $1\leq n \leq N^2 - c_q^{-2} \, N^{1+2/p} +1$, then the previous upper bound is replaced by the trivial upper bound $1$. 
\end{enumerate}
Here $c_q,c_{p,q} \in (0,1)$ are constants depending on $p$ and/or $q$ and $c\in(0,\infty)$ is an absolute constant.
\end{thmalpha}

%\bigskip
%\bigskip
%Let $1\leq p \leq 2 \leq q \leq \infty$ and assume that $n,N\in\N$ with $1\leq n \leq N^2$. Then
%\[
%c_n\big(\Sc_p^N\hookrightarrow \Sc_q^N\big) \lesssim_{q} 
%\begin{cases}
%\min\Big\{1,\frac{N^{3/2-1/p}}{n^{1/2}} \Big\} & :\, \ 1 \le n \le (1-c) N^2 \\
%N^{-1/p-1/2} (N^2 - n +1)^{1/2} & :\, \  (1-c) N^2 \le n \le N^{2}-N_q+1 \\
%N^{1/q-1/p} & :\, \  N^2 - N_q+1 \le n \leq N^2,
%\end{cases}
%\]
%where $N_q:= c N^{2/q+1}$ denotes the critical dimension and the constant $c \in (0,1)$ is the constant from Lemma \ref{lem:M}.
%\bigskip
%\bigskip

As is typical in geometric functional analysis, our proofs combine a variety of different elements, ideas, and techniques of analytic, geometric, and probabilistic flavor. These include, but are not limited to,
\begin{itemize}
%\item entropy estimates for natural embeddings of Schatten classes \cite{HPV2017},
\item duality properties and interpolation estimates for Gelfand numbers ,
%going back to the work of Gluskin \cite{G81},
\item asymptotic bounds on the expected Schatten norms of Gaussian random matrices,
\item norm estimates on suitable Schatten class subspaces of large dimension relating the Schatten norms for $1,2$ and $q$, which are based on the Dvoretzky-Milman theorem in the non-commu\-ta\-tive setting of Schatten classes,
\item a new non-commutative version of a result by V.~D. Milman on the existence of matrices with largest singular value of high multiplicity, which we couple with a comparison result due to Pietsch, 
\item the relation between norms in Schatten classes and mixed norm spaces,
\item lower bounds on Kolmogorov numbers of embeddings of mixed Lebesgue spaces due to Va\-sil'e\-va \cite{Vasil13}, and
\item a new extension of Vasil'eva's result in which we provide a lower bound on the Kolmogorov widths of mixed norm spaces and a certain set of averaged matrices.
\end{itemize}

The rest of this paper is organized as follows. In Section 2 we present the preliminaries (including notation, basic notions and background on Gelfand numbers, mixed norm spaces, and Schatten classes) and prove a variety of results needed in the proofs of Theorems \ref{thm:main sharp bounds} and \ref{thm:main non-sharp bounds}; we think some are of independent interest. In Section \ref{The case q<p}, we present the proofs of both the lower and upper bounds  in the case $0<q\le p\le \infty$. Section \ref{The case p<q} is dedicated to the proofs of lower and upper bounds in the case $0<p\le q\le \infty$.

% % % % % % % % % % % % % % % % % % % % % % % % % % % % % % % % %
\section{Preliminaries and mathematical machinery}
% % % % % % % % % % % % % % % % % % % % % % % % % % % % % % % % %

In this section we introduce our notation, central notions that appear, collect necessary background material used throughout this paper, and develop the mathematical machinery needed to prove our main results.

% % % % % % % % % % % % %
\subsection{Notation}
% % % % % % % % % % % % %

For $0< p\leq \infty$, we denote by $\ell_p^N$ the space $\R^N$ equipped with the (quasi-)norm 
\[
\big\|(x_i)_{i=1}^N\big\|_p :=
\begin{cases}
\Big(\sum\limits_{i=1}^N|x_i|^p\Big)^{1/p} & :\,  0< p < \infty \\
\max\limits_{1\leq i \leq N}|x_i| 
& :\,  p=\infty.
\end{cases}
\]
Given two quasi-Banach spaces $X,Y$, we denote by $B_X$ the closed unit ball of $X$. We shall write $\mathcal L(X,Y)$ for the space of bounded linear operators between $X$ and $Y$ equipped with the standard operator quasi-norm. For two sequences $\big(a(n)\big)_{n\in\N}$ and $\big(b(n)\big)_{n\in\N}$ of non-negative real numbers, we write $a(n)\asymp b(n)$ provided that there exist constants $c,C\in(0,\infty)$ such that $cb(n)\leq a(n)\leq Cb(n)$ for all $n\in\N$. If the constants depend on some parameter $p$, we shall write $a(n)\lesssim_p b(n)$, $a(n)\gtrsim_p b(n)$ or, if both hold, $a(n)\asymp_p b(n)$.
Similar notation is used for double sequences.

% % % % % % % % % % % % % % % % % % % % % % % % % % % % % % % % % % %
\subsection{Gelfand numbers, mixed norm spaces, and Schatten classes}
% % % % % % % % % % % % % % % % % % % % % % % % % % % % % % % % % % %
Let $X,Y$ be quasi-Banach spaces and $T\in\mathscr L(X,Y)$. For $n\in\N$, we define the $n$-th Gelfand number of the operator $T$ by
\[
c_n(T):= \inf\big\{ \|T|_F\| \,:\, F\subset X,\,\codim F < n \big\}.
\]
The operator $T$ is compact if and only if $\big(c_n(T)\big)_{n\in\N}$ converges to $0$.
Gelfand numbers belong to the more general class of $s$-numbers of operators. Those numbers are characterized by the following properties, which we shall state here for Gelfand numbers only and frequently use throughout the text:
\begin{enumerate}
\item[($G_1$)] $\|T\| = c_1(T) \geq c_2(T) \geq \dots\geq 0$ for all $T\in\mathscr L(X,Y)$.
\item[($G_2$)] $c_{m+n-1}(S+T) \leq c_m(S) + c_n(T)$ for all $S,T\in \mathscr L(X,Y)$.
\item[($G_3$)] $c_n(RST) \leq \|R\|c_n(S)\|T\|$ for all $T\in \mathscr L(X_0,Y_0)$, $S\in\mathscr L(Y_0,Y_1)$, and $R\in\mathscr L(Y_1,X_1)$.
\item[($G_4$)] If $S\in\mathscr L(X,Y)$ and $\text{rank}(S)<n$, then $c_n(S)=0$.
\end{enumerate}
Moreover, if $T$ is an isomorphism between $m$-dimensional spaces, then $c_m(T)=1/\|T^{-1}\|$. For an axiomatic approach to $s$-numbers, we refer the reader to \cite{Pietsch2} as well as the monographs \cite{CS1990} and \cite{K1986}.

The next standard tool in the geometry of Banach spaces we shall use is the concept of Kolmogorov numbers and Kolmogorov widths.
If $X,Y$ are quasi-Banach spaces
%\textcolor{red}{(J: Before we always used quasi-Banach spaces!)}
and $T\in {\mathscr L}(X,Y)$, we denote the $n$-th Kolmogorov number of $T$ by
$$
d_n(T)=\inf\big\{\|Q_N^YT\|\,:\,N\subset Y,\,\dim(N)<n\big\}.
$$
Here, $Q_N^Y$ stands for the quotient map (i.e., the natural surjection) of $Y$ onto the quotient space $Y/N$, which maps $y\in Y$ onto its equivalence class $[y]$. The definition of the $n$-th Kolmogorov number can be reformulated as follows,
$$
d_n(T)=\inf_{\substack{N\subset Y \\ \dim(N)<n}}\sup_{x\in B_X}\inf_{z\in N}\|Tx-z\|_Y.
$$
In that form, this concept can be generalized to the notion of Kolmogorov widths of sets. More precisely, if $K\subset Y$ is any subset, then the $n$-th Kolmogorov width of $K$ in $Y$ is defined as 
$$
d_n(K,Y)=\inf_{\substack{N\subset Y \\ \dim(N)<n}} \sup_{y\in K}\inf_{z\in N}\|y-z\|_Y.
$$
This means that $d_n(T)=d_n(T(B_X),Y)$. We shall later exploit the duality between Gelfand and Kolmogorov numbers \cite[Proposition 11.7.6]{Pietsch2},
which states that $c_n(T)=d_n(T^*)$ for any two Banach spaces $X,Y$ and any $T\in{\mathscr L}(X,Y)$, where $T^*$ denotes the adjoint operator from the dual
space $Y^*$ to the dual space $X^*$.

The following lemma describes the interpolation property of Gelfand numbers in a special case, see \cite[Proposition 11.5.8]{Pietsch2}. We show that the result remains valid even for quasi-Banach spaces.
%The following interpolation result is a direct consequence of the duality of Gelfand and Kolmogorov numbers (see, e.g., \cite[Lemma 2]{G81}).
%\begin{lemma}\label{lem:gluskin interpolation lemma}
%Let $X_0,X_1$, and $X_\theta$ be normed spaces with the same underlying $m$-dimensional linear space $X$ such that $\|x\|_{X_1} \leq \|x\|_{X_\theta} \le  \|x\|_{X_0}$ for all $x \in X$.
%Suppose that $\theta\in(0,1)$ is such that 
%\begin{equation}\label{eq:hoelder assumption interpolation}
%\|x^*\|_{X^*_\theta} \leq \|x^*\|_{X_0^*}^{1-\theta} \|x^*\|_{X_1^*}^\theta.
%\end{equation}
%for all $x^*\in X^*$. Then, for all $n \leq m$,
%\[
%c_n(X_1^* \hookrightarrow X_\theta^*) \leq c_n(X_1^*\hookrightarrow X_0^*)^{1-\theta}\,.
%\]
%\end{lemma} 

\begin{lemma}\label{lem:gluskin interpolation lemma}
Let $\theta\in(0,1)$ and $X_0,X_\theta,X_1$ be quasi-Banach spaces such that $X_1$ is both a subspace of $X_0$ and $X_\theta$ and such that $\|x\|_{X_0}\le \|x\|_{X_\theta}\le \|x\|_{X_1}$ for $x\in X_1$.
If
\begin{equation}\label{eq:hoelder assumption interpolation}
\|x\|_{X_\theta} \leq \|x\|_{X_0}^{1-\theta} \|x\|_{X_1}^\theta
\end{equation}
for all $x\in X_1$, then
\[
c_n(X_1 \hookrightarrow X_\theta) \leq c_n(X_1\hookrightarrow X_0)^{1-\theta}
\]
for all $n\in\N$.
\end{lemma}
\begin{proof}
 Let $n\in\N$ and $\varepsilon>0$ and choose a subspace $F$ of $X_1$ with $\codim F < n$ such that
 \[
  \|x\|_{X_0} \le (1+\varepsilon) c_n(X_1\hookrightarrow X_0)\|x\|_{X_1}
 \]
 for all $x\in F$. Now \eqref{eq:hoelder assumption interpolation} implies for all $x\in F$ that
 \[
 \|x\|_{X_\theta} \leq \|x\|_{X_0}^{1-\theta} \|x\|_{X_1}^\theta \leq (1+\varepsilon)^{1-\theta} c_n(X_1\hookrightarrow X_0)^{1-\theta} \|x\|_{X_1}.
 \]
 Hence 
 \[
 c_n(X_1 \hookrightarrow X_\theta) \leq (1+\varepsilon)^{1-\theta} c_n(X_1\hookrightarrow X_0)^{1-\theta}
 \] 
 holds for every $\varepsilon>0$. This proves the Lemma.
\end{proof}

The singular values $s_1,\dots,s_N$ of a real $N\times N$ matrix $A$ are defined to be the square roots of the eigenvalues of the positive self-adjoint operator $A^*A$, which are simply the eigenvalues of $|A|:=\sqrt{A^*A}$. The singular values are arranged in non-increasing order, that is, $s_1(A) \geq \dots \geq s_N(A)\geq 0$. The singular value decomposition shall be used in the form $A=U\Sigma V^T$, where $U,V\in\R^{N\times N}$ are orthogonal matrices,
and $\Sigma\in\R^{N\times N}$ is a diagonal matrix with $s_1(A),\dots,s_N(A)$ on the diagonal.

For $0<p\le \infty$, the Schatten $p$-class $\Sc_p^N$ is the $N^2$-dimensional space of all $N\times N$ real matrices acting from $\ell_2^N$ to $\ell_2^N$ equipped with the Schatten $p$-norm
\[
\|A\|_{\Sc_p}=\bigg(\sum_{j=1}^Ns_j(A)^p\bigg)^{1/p}.
\]
Let us remark that $\|\cdot\|_{\Sc_1}$ is the nuclear norm, $\| \cdot \|_{\Sc_2}$ the Hilbert-Schmidt norm, and $\|\cdot\|_{\Sc_\infty}$ the operator norm. We denote the unit ball of $\Sc_p^N$ by
\begin{align*}
B_p^N :=\Big\{A\in\R^{N\times N}\,:\,\|A\|_{\Sc_p}\le 1\Big\}.
\end{align*}

We will also use a connection between Schatten norms of matrices and the so-called iterated (or mixed) norms. For $0<p,q<\infty$, we define the mixed (quasi-)norm of a matrix $M=(M_{j,k})_{j,k=1}^N\in\R^{N\times N}$ by
$$
\|M\|_{\ell_q(\ell_p)}=\Biggl(\sum_{k=1}^N \Bigl(\sum_{j=1}^N|M_{j,k}|^p\Bigr)^{q/p}\Biggr)^{1/q}
$$
with the usual modification if $p=\infty$ and/or $q=\infty$.
Note that $\|M\|_{\Sc_2}=\|M\|_{\ell_2(\ell_2)}$ and that
\begin{align*}
\|M\|_{\Sc_\infty}&=\sup_{\|x\|_2\le 1}\|Mx\|_{\ell_2}\ge \max_{1\leq k \leq N}\|Me_k\|_{\ell_2}
=\max_{1\leq k \leq N}\Bigl(\sum_{j=1}^N|M_{j,k}|^2\Bigr)^{1/2}=\|M\|_{\ell_\infty(\ell_2)}.
\end{align*}
By interpolation properties of vector-valued sequence spaces \cite[Theorem 5.6.1]{BL} and of
Schatten classes \cite[Theorem 2.3.14]{Pietsch1987}
%\textcolor{red}{(J: we should elaborate on the following bounds, not just say interpolation and duality!
%This is unclear and also no references for duality or interpolation results are given above!
%We might include those before discussing the mixed norms.)},
we therefore obtain the bound
\begin{align}
\label{eq:pq1}\|M\|_{\ell_q(\ell_2)}&\le \|M\|_{\Sc_q}
\end{align}
whenever $2\le q\le \infty$. Moreover, by duality (see \cite[Lemma 1.11.1]{TRI} and \cite[Proposition IV.2.11]{Bhatia}), we also get
\begin{align}
\label{eq:pq2}\|M\|_{\Sc_p}&\le \|M\|_{\ell_p(\ell_2)}
\end{align}
whenever $1\le p\le 2$.

% % % % % % % % % % % % % % % % % %
\subsection{Matrices with largest singular value of high multiplicity}
% % % % % % % % % % % % % % % % % %

In order to obtain sharp lower bounds for the Gelfand numbers in the regime $0<q\le p\le \infty$, we prove a non-commutative version of a result due to V.~D. Milman \cite{Milman1970} and combine this with a comparison estimate going back to Pietsch \cite[Lemma 2.9.7]{Pietsch1987}.

The comparison result of Pietsch is contained in the following Lemma. In fact, by following verbatim the proof presented in \cite[Lemma 2.9.7]{Pietsch1987}, one can see that the result carries over to the quasi-Banach space setting and we therefore omit the details.

\begin{lemma}\label{lem:pietsch}
Let $0<q<p\leq \infty$, $m\in\N$, and assume that $x=(x_1,\dots,x_{m+1})\in\R^{m+1}$. If 
\[
|x_{m+1}| \leq \min\big\{|x_1|,\dots,|x_m| \big\},
\]
then we have
\[
\frac{\Big(\sum_{\ell=1}^{m+1}|x_\ell|^q\Big)^{1/q}}{\Big(\sum_{\ell=1}^{m+1}|x_\ell|^p\Big)^{1/p}} \geq \frac{\Big(\sum_{\ell=1}^{m}|x_\ell|^q\Big)^{1/q}}{\Big(\sum_{\ell=1}^{m}|x_\ell|^p\Big)^{1/p}}\,.
\]
\end{lemma}
 
It was already noted in \cite{HM2005} that a non-commutative analogue of the result of V.~D. Milman  would extend the lower bound obtained there to the whole range of parameters $0<q\le p\le \infty$.
As a matter of fact, while we can essentially directly apply the result of Pietsch (slightly generalized to the quasi-Banach space setting), it seems that the non-commutative version of Milman's result cannot be obtained along the original lines via an extreme point argument coupled with the Krein-Milman theorem. The reason is that the geometry and extreme point structure of the unit ball in $\Sc_\infty^N$ is more complicated than the one of $[-1,1]^N$, ultimately making it a more delicate argument. We overcome this problem by taking a different approach which is motivated by \cite{FL1976}. 

The following lemma guarantees the existence of matrices with $k$ singular values equal to one in any subspace of dimension larger than or equal to some number $\kappa=\kappa(k)$. As explained above, it can be considered a non-commutative version of a result of V.~D. Milman. The proof rests on a subspace splitting and perturbation argument.

\begin{lemma}\label{lem:qp} Let $k,n,N\in\N$ and assume that $1\le k\le N$ and $N^2\ge n\ge \varkappa(k):=(2N-k+1)(k-1)+1$.
Let $S\subset \R^{N\times N}$ be a linear subspace with
$\dim S = n$. Then there exists a matrix $A\in S$ such that 
\[ 
 \|A\|_{\Sc_\infty}=\sigma_1(A)=\dots=\sigma_k(A)=1.
\]
\end{lemma}
\begin{proof} First, we assume that $k=1$. Then $ \varkappa(1)=1$ and for every $n\in\N$ with $1\leq n\leq N^2$ the statement holds true by appropriate normalization of any non-zero matrix. 

Next we assume that the statement
holds for $1\le k \le N-1$ and prove it for $k+1$. Note that $\kappa$ is a monotone increasing function of $k$. So let us assume that $n=\varkappa(k+1)=(2N-k)k+1\ge \varkappa(k)$ and that $S\subset \R^{N\times N}$ is a linear subspace of dimension $n$. By assumption there exists a matrix $B\in S$ so that
$\|B\|_{\Sc_\infty}=\sigma_1(B)=\dots=\sigma_k(B)=1$. Let $B=U\Sigma V^T$ be the singular value decomposition of the matrix $B$, where $U,V\in\R^{N\times N}$ are orthogonal matrices, $V^T$ is the transpose of $V$, and $\Sigma$ the diagonal matrix containing the singular values of $B$.
We denote by $u_j\in\R^N$ and $v_j\in\R^N$, $j=1,\dots,N$ the column vectors of $U$ and $V$, respectively.
We now introduce the subspaces 
\[
U_{k}^{-}:=\linspan\{u_1,\dots,u_k\} \qquad\text{and}\qquad U_k^{+}:=\linspan\{u_{k+1},\dots,u_N\}=(U_k^{-})^\perp
\]
and similarly the spaces $V_k^{-}:=\linspan\{v_1,\dots,v_k\}$, $V_k^{+}:=\linspan\{v_{k+1},\dots,v_N\}=(V_k^{-})^\perp$. Furthermore,
let $X_1,\dots,X_n$ be a basis of the linear subspace $S$. We consider a matrix $X=\sum_{j=1}^n\alpha_j X_j$, $\alpha_1,\dots,\alpha_n\in\R$ and the conditions
\begin{align}
\label{eq:lemspace:1}Xv_j=0,&\qquad j=1,\dots,k\\
\label{eq:lemspace:2}Xv_j\perp u_\ell,&\qquad j=k+1,\dots,N\,\,\, \text{and}\,\,\, \ell=1,\dots,k.
\end{align}
These conditions can be rewritten as linear conditions on $\alpha=(\alpha_1,\dots,\alpha_n)$ as
\begin{align*}
\sum_{m=1}^n \alpha_m X_m v_j=0,&\qquad j=1,\dots,k\\
\sum_{m=1}^n \alpha_m \langle X_m v_j,u_\ell\rangle=0,&\qquad j=k+1,\dots,N\,\,\, \text{and}\,\,\, \ell=1,\dots,k.
\end{align*}
They pose $kN+(N-k)k=2kN-k^2=n-1$ linear conditions on $\alpha$ and therefore there is a non-trivial solution $\alpha\in\R^n.$
We then consider the corresponding $X=\sum_{j=1}^n\alpha_j X_j$ and the family of matrices
$$
A_\gamma=\gamma X+B,\quad \gamma\in\R.
$$
The key idea is that under \eqref{eq:lemspace:1} and \eqref{eq:lemspace:2} this additive perturbation of the matrix $B$, which already has singular values $\sigma_1(B)=\dots=\sigma_k(B)=1$, does not influence the column subspace $V_k^{-}$ and only acts non-trivially on the space $V_k^{+}$, which means that the rank of $A_\gamma$ is at least $k+1$, i.e., $A_\gamma$ has at least $k+1$ non-zero singular values. In particular, this separation of influence of the different summands in the definition of $A_\gamma$ allows us to choose an appropriate $\gamma$ such that we obtain $k+1$ singular values of $A_\gamma$ equal to 1. More precisely, the two conditions \eqref{eq:lemspace:1} and \eqref{eq:lemspace:2} ensure that $A_\gamma(V_k^{-})=U_k^{-}$ and $A_\gamma(V_k^{+})\subset U_k^{+}$ and hence
we may choose $\gamma\in\R$ such that
$$
\|A_{\gamma}\|_{\Sc_\infty}=\sigma_1(A_\gamma)=\dots=\sigma_{k+1}(A_\gamma)=1.
$$
This concludes the proof.
\end{proof}

\begin{rmk} Let us remark that, for $1\le k \le N$,
$$
N^2-\varkappa(k)=N^2-2N(k-1)+(k-1)^2-1=(N-k)(N-k+2)\ge 0
$$
and, therefore, $1\le \varkappa(k)\le N^2$.
\end{rmk}

\subsection{Dvoretzky's Theorem for Schatten classes}

The following result is essential to some of our arguments and based on V.~D. Milman's version of Dvoretzky's Theorem, which enters the proof in both its Gaussian and geometric version (see, e.g., \cite[Theorem 5.4.4]{IsotropicConvexBodies}).

\begin{lemma}
	\label{lem:M}
	There exists a constant $c\in(0,1)$ such that for all
	$q\in[2,\infty]$ and $k\leq c N^{1+2/q}$ there exists a subspace \(L\subset\Sc_q^N\) with $\dim L\ge k$ such that, for all $A\in L$,
		\begin{equation}\label{eq:Dvoretzky0}
			c_1(q)^{-1} N^{1/2-1/q} \|A\|_{\Sc_q}\le\|A\|_{\Sc_2}\le c_2 N^{-1/2}\|A\|_{\Sc_1}.
		\end{equation}
		with $c_1(q)\in(0,\infty)$ being monotone increasing in $q$.
\end{lemma}

Before we present the proof, let us recall the Dvoretzky-Milman theorem (see \cite{IsotropicConvexBodies} and \cite{FLM1977}). For $n\in\N$, we shall denote by $G_n$ a standard Gaussian random vector in $\R^n$. 

\begin{proposition}[Dvoretzky-Milman]\label{prop:dvoretzky-milman probabilistic}
Let $\|\cdot\|$ be a norm on $\R^n$, $b:=\max\{\|t\|\,:\, t\in\SSS^{n-1} \}$. For any $\varepsilon>0$ there exists a constant $c(\varepsilon)\in(0,\infty)$ such that for any integer $1\leq k \leq c(\varepsilon) (\E\|G_n\|/b)^2$ there exists a linear subspace $E$ of dimension $k$ such that, for all $t\in E$,
\[
(1-\varepsilon)\|t\|_2 \leq \|t\|\cdot \frac{\E\|G_n\|_2}{\E\|G_n\|} \leq (1+\varepsilon)\|t\|_2.
\] 
\end{proposition}

\begin{rmk}\label{rem:geometric dvoretzky-milman}
\begin{enumerate}
\item For our purposes, we shall simply use $\varepsilon=1/2$ and therefore do not require the full strength of the Dvoretzky-Milman result. 
\item Proposition \ref{prop:dvoretzky-milman probabilistic} can be stated in a ``geometric'' form: denote by $M:=M_{\|\cdot\|}$ the median of $\|\cdot\|$ on $\SSS^{n-1}$,
i.e.,
\[
\sigma\Big(\Big\{x\in\SSS^{n-1}\,:\, \|x\|\geq M \Big\}\Big) \geq \frac{1}{2}
\qquad
\text{and}
\qquad
\sigma\Big(\Big\{x\in\SSS^{n-1}\,:\, \|x\|\leq M \Big\}\Big) \geq \frac{1}{2},
\]
where $\sigma$ is the normalized surface measure on $\SSS^{n-1}$. Then there exists a linear subspace $E$ of dimension $k\geq c n (M/b)^2$ such that, for all $t\in E$,
\[
\frac{1}{2} \|t\|_2 \leq \frac{\|t\|}{M} \leq 2 \|t\|_2. 
\]
\end{enumerate}
\end{rmk}

The following lemma provides asymptotic bounds for the expectation of Schatten $q$-norms of Gaussian random matrices for all $1\leq q \leq \infty$ and will also be used in the proof of Lemma \ref{lem:M}. 

\begin{lemma}\label{lem:mean_values_S_q} Let $1\leq q \leq \infty$, $N\in\N$, and $G$ be an $N\times N$ Gaussian random matrix with independent $\mathcal N(0,1)$ entries. Then
\begin{align}\label{eq:equivalence expected value gaussian matrix in schatten norm}
\E\|G\|_{\Sc_q} & \asymp N^{1/2+1/q}
\end{align}
with the constants of equivalence indepenent of $N$ and $q$.
\end{lemma}
\begin{proof}
%\vskip 1mm
\noindent\emph{Upper bound}:\,\, By \cite[Theorem 5.32]{Vershynin}, we know that there exists an absolute constant $C\in(0,\infty)$ such that
\begin{equation}\label{eq:ESinfty}
\E \|G\|_{\Sc_\infty} \leq C N^{1/2}.
\end{equation}
%where $C\in(0,\infty)$ is an absolute constant.
The upper bound in \eqref{eq:equivalence expected value gaussian matrix in schatten norm}
then follows from \eqref{eq:ESinfty} combined with H\"older's inequality, i.e.,
\begin{align}\label{eq:upper bound expectation schatten q norm gaussian matrix}
\E\|G\|_{\Sc_q} \leq N^{1/q} \E \|G\|_{\Sc_\infty}\le CN^{1/2+1/q}\,.
\end{align}
\vskip 1mm
\noindent\emph{Lower bound}:\,\,It follows (in that order) from \cite[Corollary 3.2]{Ledoux} applied with $p=2$ and $q=1$ there,
H\"older's inequality, the Cauchy-Schwarz inequality, and the estimate \eqref{eq:upper bound expectation schatten q norm gaussian matrix} applied to $1\leq q^* \leq \infty$ that
\[
N = \sqrt{\E \|G\|_{\Sc_2}^2}\lesssim \E\|G\|_{\Sc_2}\le \E \left[\|G\|^{1/2}_{\Sc_q}\|G\|^{1/2}_{\Sc_{q^*}}\right]
\le \sqrt{\E \|G\|_{\Sc_q}}\,\sqrt{\E\|G\|_{\Sc_{q^*}}}\lesssim N^{1/4+1/(2q^*)}\sqrt{\E \|G\|_{\Sc_q}}\,,
\]
where $q^*$ is the H\"older conjugate of $q$.
%The last inequality follows from the upper bound for $1\leq q^* \leq \infty$ proved above.
Rearranging the previous bound, we obtain
\[
\E \|G\|_{\Sc_q}\gtrsim N^{2-1/2-1/q^*}=N^{3/2-(1-1/q)}=N^{1/2+1/q}.
\]
This shows the equivalence in \eqref{eq:equivalence expected value gaussian matrix in schatten norm}.
\end{proof}

We can now present the proof of Lemma \ref{lem:M}, which iteratively uses the matrix versions of the probabilistic and geometric Dvoretzky-Milman
result to find a good subspace that later allows us to bound from above the Gelfand numbers of Schatten class embeddings.

\begin{proof}[Proof of Lemma \ref{lem:M}]
We first compute some parameters that appear in Proposition \ref{prop:dvoretzky-milman probabilistic}.
Let $1\leq q \leq \infty$ and consider $\|\cdot\|_{\Sc_q}$ on $\R^{N\times N}$. We start with the computation of $b$:
\begin{align}\label{eq:computation of parameter b}
b & = \max_{\|A\|_{\Sc_2} \leq 1} \|A\|_{\Sc_q} = \| \Sc_2^N \hookrightarrow \Sc_q^N \| = 
\begin{cases}
N^{1/q-1/2} & :\, 1\leq q\leq 2 \\
1 & :\, 2\leq q\leq \infty.
\end{cases}
\end{align}
Further, we consider a Gaussian random matrix $G:=G_N := (g_{ij})_{i,j=1}^N$, where $g_{ij}$, $i,j\in\{1,\dots,N\}$ are independent standard Gaussian random variables.
It follows from Lemma \ref{lem:mean_values_S_q} that
\[
\frac{\E \|G\|_{\Sc_2}}{\E\|G\|_{\Sc_q}}\asymp \frac{N}{N^{1/2+1/q}}=N^{1/2-1/q}.
\]

After having computed the parameters, we continue with the Dvoretzky argument.
In view of Proposition \ref{prop:dvoretzky-milman probabilistic} (with $\varepsilon=1/2$), we have shown that for any $k\in\N$ with
\[
1\leq k \leq 
\begin{cases}
C_1\,N^2& : 1\leq p \leq 2 \\
C_2\,N^{2/p+1} & : 2\leq p \leq \infty,
\end{cases}  
\]
there exists a subspace $L\subset \Sc_p^N$ such that, for all $A\in L$,
\[
\frac{1}{2} \|A\|_{\Sc_2} \leq \|A\|_{\Sc_p}\,N^{1/2-1/p} \leq \frac{3}{2}\|A\|_{\Sc_2}\,.
\]
It 
%is well known to experts and 
was shown by Figiel, Lindenstrauss, and V.~D. Milman \cite[Example 3.3]{FLM1977} that the subspace dimensions are in fact optimal. We now iteratively use the Dvoretzky-Milman result for Schatten classes to obtain the desired subspace together with the norm estimates.
\vskip 1mm
\noindent\emph{Step 1.} We choose $p=1$. Then there exists a subspace $L_1\subset \Sc_1^N$ with $\dim L_1 \gtrsim N^2$ such that, for all $A\in L_1$,
\begin{align}\label{eq:norm estimate on L1}
\frac{1}{2}\|A\|_{\Sc_2} \leq \|A\|_{\Sc_1} N^{-1/2}\,.
\end{align}
Algebraically, we consider $L_1$ to be a subspace of $\Sc_q^N$.
\vskip 1mm
\noindent\emph{Step 2.} We now apply the geometric version of the Dvoretzky-Milman theorem to $L_1$ equipped with the $\Sc_q$-norm
(see Remark \ref{rem:geometric dvoretzky-milman}). Observe that, since $q\geq 2$,
\[
b_{L_1} = \max_{ \substack { A\in L_1 \\ \|A\|_{\Sc_2}=1}}\|A\|_{\Sc_q} \leq \max_{ \substack{A\in\Sc_2^N \\ \|A\|_{\Sc_2}=1}} \|A\|_{\Sc_q} = \|\Sc^N_2\hookrightarrow\Sc^N_q\| \leq 1\,.
\]
To estimate the median, we use $\|A\|_{\Sc_q} \geq N^{1/q-1/2}\|A\|_{\Sc_2}$ and obtain that
\begin{align}\label{eq:estimate ML1 below}
M_{L_1,\|\cdot\|_{\Sc_q}} \geq N^{1/q-1/2}\,M_{L_1,\|\cdot\|_{\Sc_2}} = N^{1/q-1/2}.
\end{align}
Hence, there exists a subspace $L_2\subset L_1$ with 
\begin{equation}\label{eq:new1}
\dim L_2 \gtrsim \dim L_1 \cdot \Bigg(\frac{M_{L_1,\|\cdot\|_{\Sc_q}}}{b_{L_1}}\Bigg)^2 \gtrsim N^{2/q+1}
\end{equation}
such that, for all $A\in L_2$,
\begin{align}\label{eq:norm estimate on L2}
\frac{1}{2} \|A\|_{\Sc_2} \leq \frac{\|A\|_{\Sc_q}}{M_{L_1,\|\cdot\|_{\Sc_q}}} \leq 2\|A\|_{\Sc_2}\,.
\end{align}
Since, as we mentioned before, the subspace dimensions are optimal in the case of Schatten classes, we know that $\dim L_2 \lesssim_q N^{2/q+1}$.
% and %thus, since
%\[
%\dim L_2 \gtrsim \dim L_1 \cdot \Bigg(\frac{M_{L_1,\|\cdot\|_{\Sc_q}}}{b_{L_1}}\Bigg)^2,
%\]
By \eqref{eq:new1} we obtain
\[
M_{L_1,\|\cdot\|_{\Sc_q}} \lesssim \Bigg(\frac{\dim L_2}{\dim L_1}\Bigg)^{1/2}\, b_{L_1} \leq C_q\cdot \frac{N^{1/q+1/2}}{N} = C_q\, N^{1/q-1/2},
\]
where $C_q\in(0,\infty)$ depends only on $q$. Because of the lower bound obtained in \eqref{eq:estimate ML1 below}, this means that 
\begin{align}\label{eq:asymptotic of ML1}
N^{1/q-1/2} \leq M_{L_1,\|\cdot\|_{\Sc_q}} \leq C_q\, N^{1/q-1/2}\,.
\end{align}
Note that because $\|\cdot\|_{\Sc_q}N^{1/2-1/q}$ is a monotone increasing sequence in $q$, also $M_{L_1,\|\cdot\|_{\Sc_q}}N^{1/2-1/q}$ is increasing and, therefore, the
constant $C_q$ above may be chosen to increase in $q$ as well.

Now, combining \eqref{eq:norm estimate on L2} with \eqref{eq:asymptotic of ML1}, we obtain that, for all $A\in L_2$,
\begin{align}\label{eq:norm estimates on L2 with median replaced by its asymptotics}
\frac{1}{2}N^{1/q-1/2} \|A\|_{\Sc_2} \leq \|A\|_{\Sc_q} \leq 2 C_q N^{1/q-1/2} \,\|A\|_{\Sc_2}\,.
\end{align}
\vskip 1mm
\noindent\emph{Step 3.} Putting everything together, in particular combining the estimates \eqref{eq:norm estimate on L1} and \eqref{eq:norm estimates on L2 with median replaced by its asymptotics},
we obtain the following: for every $k\lesssim N^{2/q+1}$ there exists a subspace $L_2\subset \Sc_q^N$ with $\dim L_2 \geq k$ such that, for all
 $A\in L_2$,
\[
\frac{1}{2C_q}\,N^{1/2-1/q} \|A\|_{\Sc_q} \leq \|A\|_{\Sc_2} \leq 2N^{-1/2}\|A\|_{\Sc_1}\,.
\]
This completes the proof of the Lemma.
\end{proof}

After having set-out the mathematical machinery, we are now going to present the proofs of our main results in the next two sections.

% % % % % % % % % % % % % % % % % % % % % % %
\section{The case $0<q\le p\le \infty$} \label{The case q<p}
% % % % % % % % % % % % % % % % % % % % % % % 

% % % % % % % % % % % % % % % % % % % % % % % % % % % % % % % %
\subsection{The upper bound in the case \(0< q \le p <\infty\)}
% % % % % % % % % % % % % % % % % % % % % % % % % % % % % % % %

As in the case of $\ell_p$ sequence spaces it is rather easy to find a subspace of the right codimension to prove a sharp upper bound. The subspaces giving the exact value of $c_n(\ell_p^N\hookrightarrow \ell_q^N)$ for $q<p$ are just the coordinate subspaces. This directly hints to using subspaces containing only matrices with few nonzero singular values to estimate $c_n(\Sc_p^N\hookrightarrow \Sc_q^N)$. This is what we do in the proof of the next Proposition. We mention that a statement of this upper bound can be found in \cite[Example 4.7 (ii)]{HM2005} for the case $1\le q\le 2$.

\begin{proposition}
Let $0<q\le p\le \infty$ and assume that $n,N\in\N$ with $1\leq n \leq N^2$. Then
\[
c_n\big(\Sc_p^N\hookrightarrow \Sc_q^N\big) \lesssim_{p,q} 
\begin{cases}
\bigg(\frac{N^2 - n +1}{N} \bigg)^{1/q-1/p} & :\, 1\leq n \leq N^2-N+1 \\
1 & :\, N^2-N+1 \leq n \leq N^2. 
\end{cases}
\]
\end{proposition}

\begin{proof}
Let $M$ be the linear subspace of all matrices in $\R^{N\times N}$ having all entries in the first $k$ rows equal to $0$.
This subspace has codimension $k N$. Since all matrices in $M$ have rank at most $N-k$, the number of nonzero singular values of such a matrix is at most $N-k$.
H\"older's inequality then implies that 
\[
 \|A\|_{\Sc_q} \leq (N-k)^{1/q-1/p} \|A\|_{\Sc_p} 
\]
for any matrix $A\in M$.
Hence, we obtain 
\[
c_n\big(\Sc_p^N\hookrightarrow \Sc_q^N\big)
\leq (N-k)^{1/q-1/p}
\]
whenever $k N < n$. 
Choosing $k=\left\lfloor \frac{n-1}{N} \right\rfloor$, we have $N-k = \left\lceil \frac{N^2-n+1}{N} \right\rceil$ and obtain
\[
c_n\big(\Sc_p^N\hookrightarrow \Sc_q^N\big)
\leq \left\lceil \frac{N^2-n+1}{N} \right\rceil^{1/q-1/p}.
\]
This easily translates into the claimed estimate.
\end{proof}

\subsection{The lower bound in the case $0<q\le p\le \infty$}
% % % % % % % % % % % % % % % % % % % % % % % 

We shall now continue with the lower bounds in the regime $0<q\le p\le \infty$. 

\begin{proposition} Let $0<q\le p\le \infty$ and assume that $n,N\in\N$ with $1\leq n \leq N^2$. Then
\[
c_n\big(\Sc_p^N\hookrightarrow \Sc_q^N\big) \gtrsim_{p,q} 
\begin{cases}
\bigg(\frac{N^2 - n +1}{N} \bigg)^{1/q-1/p} & :\, 1\leq n \leq N^2-N+1 \\
1 & :\, N^2-N+1 \leq n \leq N^2. 
\end{cases}
\]
\end{proposition}
\begin{proof} 
\noindent\emph{Step 1.} First, let $N^2-N+1 \leq n \leq N^2$.
Then, since $q\leq p$, the claimed estimate follows from the monotonicity of the Gelfand numbers 
\[ 
c_n\big(\Sc_p^N\hookrightarrow \Sc_q^N\big)
\geq
c_{N^2}\big(\Sc_p^N\hookrightarrow \Sc_q^N\big)
=
\| \Sc_q^N\hookrightarrow \Sc_p^N \|^{-1}
=
1.
\]
\vskip 1mm
\noindent\emph{Step 2.} Now consider $1\leq n \leq N^2-N+1$.
Let $S\subset \R^{N\times N}$ be a linear subspace with $\codim S< n.$ Then we may assume that $\codim S=n-1$ and thus
$\dim S=N^2-n+1.$ We choose $k\in\N$ with $k\le N$ such that
\begin{equation}\label{eq:choice of k}
N^2-n+1\ge \varkappa(k):=(2N-k+1)(k-1)+1 = 2N(k-1) - (k-1)^2 + 1,
\end{equation}
where $\varkappa(k)$ is the same as in Lemma \ref{lem:qp}.
Although the optimal (i.e., the largest) $k$ could be found easily, we simply take the largest $k$ such that
$$
N^2-n\ge 2N(k-1),
$$
which then in particular yields \eqref{eq:choice of k}.
Then
$$
k\le \frac{N^2-n}{2N}+1\qquad \text{and}\qquad k\ge\frac{N^2-n+1}{2N},
$$
where the latter holds due to the maximality of $k$.
Using Lemma \ref{lem:qp}, we find that there exists $A\in S$ with $\|A\|_{\Sc_\infty}=\sigma_1(A)=\dots=\sigma_k(A)=1.$
By the comparison Lemma \ref{lem:pietsch} applied to the sequence of singular values of $A$, we obtain
\begin{align*}
c_n \bigl(\Sc_p^N\hookrightarrow \Sc_q^N\bigr)&\ge \frac{\|A\|_{\Sc_q}}{\|A\|_{\Sc_p}}
=\frac{\displaystyle\biggl(\sum_{\ell=1}^N\sigma_\ell(A)^q\biggr)^{1/q}}{\displaystyle\biggl(\sum_{\ell=1}^N\sigma_\ell(A)^p\biggr)^{1/p}}
\ge \frac{\displaystyle\biggl(\sum_{\ell=1}^k\sigma_\ell(A)^q\biggr)^{1/q}}{\displaystyle\biggl(\sum_{\ell=1}^k\sigma_\ell(A)^p\biggr)^{1/p}}\\
&= k^{1/q-1/p}\ge \biggl(\frac{N^2-n+1}{2N}\biggr)^{1/q-1/p}.
\end{align*}
This proves the result.
\end{proof}

% % % % % % % % % % % % % % % % % % % % % % %
\section{The case $0<p\le q\le \infty$}\label{The case p<q}
% % % % % % % % % % % % % % % % % % % % % % % 

% % % % % % % % % % % % % % % % % % % % % % % % % % % % % % % %
\subsection{The upper bound in the case \(1\leq p\leq 2\le q\le\infty\)}
% % % % % % % % % % % % % % % % % % % % % % % % % % % % % % % %

The upper bound in this regime can be obtained from a Dvoretzky argument within the framework of Schatten classes. The same idea has been used by Gluskin in \cite{G81} for the classical $\ell_p$-spaces.

The next proposition provides upper bounds on the Gelfand numbers in the regime $1\leq p \leq 2 \leq q \leq \infty$. This bound complements Proposition \ref{prop: lower bound 1 < p <2 <q<infty} below, showing that the bound is in fact sharp when $1\leq n \leq c_{p,q}N^2$ for some constant $c_{p,q}\in(0,1)$ depending only on $p$ and $q$. Moreover, the same proposition also shows sharpness of the bound for  $ N^2 - N_q+1 \le n \leq N^2$.

\begin{proposition}\label{prop:upper bound for 1 leq p leq 2 leq q leq infty}
Let $1\leq p \leq 2 \leq q \leq \infty$ and assume that $n,N\in\N$ with $1\leq n \leq N^2$. Then
\[
c_n\big(\Sc_p^N\hookrightarrow \Sc_q^N\big) \lesssim_{q} 
\begin{cases}
\min\Big\{1,\frac{N^{3/2-1/p}}{n^{1/2}} \Big\} & :\, \ 1 \le n \le (1-c) N^2 \\
N^{-1/p-1/2} (N^2 - n +1)^{1/2} & :\, \  (1-c) N^2 \le n \le N^{2}-N_q+1 \\
N^{1/q-1/p} & :\, \  N^2 - N_q+1 \le n \leq N^2,
\end{cases}
\]
where $N_q:= c N^{2/q+1}$ denotes the critical dimension and the constant $c \in (0,1)$ is the constant from Lemma \ref{lem:M}.
\end{proposition}
\begin{proof}
\noindent\emph{Step 1.}
First, we consider the case $N^2 - N_q+1 \le n \leq N^2$ of large codimension $n$. 
The result is that the Gelfand numbers $c_n\big(\Sc_p^N\hookrightarrow \Sc_q^N\big)$ in this range are comparable to the last Gelfand number
\[
c_{N^2}\big(\Sc_p^N\hookrightarrow \Sc_q^N\big) = \frac{1}{\big\|\Sc_q^N\hookrightarrow \Sc_p^N\big\|} = N^{1/q-1/p}.
\]
Lemma \ref{lem:M} applied to $k:=N^2-n+1$, which in our range means $k\leq N_q=cN^{2/q+1}$, implies that there exists a subspace \(L\subset\Sc_q^N\) with $\dim L \geq k$ such that, for all $A \in L$, we have
\begin{equation*}
c_1(q )^{-1} N^{1/2-1/q} \|A\|_{\Sc_q}\le\|A\|_{\Sc_2}\le c_2 N^{-1/2}\|A\|_{\Sc_1}\, .
\end{equation*}
In particular, this implies that, for all $A\in L$,
\[
 \|A\|_{\Sc_q} 
   \le c_1(q)c_2 N^{1/q-1} \|A\|_{\Sc_1} 
   \le c_1(q)c_2 N^{1/q-1} \big\|\Sc_p^N\hookrightarrow \Sc_1^N\big\| \, \|A\|_{\Sc_p}  
   = c_1(q)c_2 N^{1/q-1/p} \|A\|_{\Sc_p} .
\]
Combining the latter bound with the observation that $\codim L = N^2-\dim L \leq N^2-(N^2-n+1) = n-1 <n $, and algebraically considering $L$ as a subspace of $\Sc_p^N$, we obtain
\[
 c_n\big(\Sc_p^N\hookrightarrow \Sc_q^N\big) \le c_1(q) c_2\, N^{1/q-1/p}.
\]
\vskip 1mm
\noindent\emph{Step 2.}
Now we consider the case $(1-c) N^2 < n \le N^{2}-N_q+1$ of medium codimension $n$ and apply a lifting argument. The main idea behind it is to introduce an intermediate value $s \in (2,q)$ such that $n=N^{2}-N_s+1$ and to use the result of Step 1 for $c_n\big(\Sc_p^N\hookrightarrow \Sc_s^N\big)$.
In order for such an $s > 2$ to exist we require the condition  $(1-c) N^2 < n$.
Then we obtain
\[
      c_n\big(\Sc_p^N\hookrightarrow \Sc_q^N\big)
  \le c_n\big(\Sc_p^N\hookrightarrow \Sc_s^N\big)
  \le c_1(s) c_2 N^{1/s-1/p}
  \le c_1(q) c_2 N^{1/s-1/p}.
\]
The last inequality follows again from Lemma \ref{lem:M}.
Now the equation $n=N^{2}-N_s+1=N^2 - c N^{2/s+1}+1$ is easily transformed into the equation $N^{1/s} = c^{-1/2} N^{-1/2} (N^2-n+1)^{1/2}$. Altogether, we arrive at the claimed estimate
\[
      c_n\big(\Sc_p^N\hookrightarrow \Sc_q^N\big)
  \le c_1(q) c_2 N^{1/s-1/p}
  \le c_1(q) c_2 \, c^{-1/2} \, N^{-1/2-1/p} (N^2-n+1)^{1/2}.
\]
\vskip 1mm
\noindent\emph{Step 3.}
Finally, the case $1 \le n \le (1-c) N^2$ is settled using the Carl-Defant result \eqref{eq:CD} and factorization. From this, we obtain the estimate
\[
      c_n\big(\Sc_p^N\hookrightarrow \Sc_q^N\big)
  \le \big\|\Sc_p^N\hookrightarrow \Sc_1^N\big\| \,
      c_n\big(\Sc_1^N\hookrightarrow \Sc_2^N\big) \,
      \big\|\Sc_2^N\hookrightarrow \Sc_q^N\big\| 
  \lesssim N^{1-1/p} \, \left( \frac{N}{n} \right)^{1/2} 
  = \frac{N^{3/2-1/p}}{n^{1/2}},
\]
which, together with the trivial estimate 
\[
  c_n\big(\Sc_p^N\hookrightarrow \Sc_q^N\big)
  \le \big\|\Sc_p^N\hookrightarrow \Sc_q^N\big\|
  = 1,
\]
completes the proof.
\end{proof}

% % % % % % % % % % % % % % % % % % % % % % % % % % % % % % % %
\subsection{The upper bound in the case \(2< p < q\le\infty\)}
% % % % % % % % % % % % % % % % % % % % % % % % % % % % % % % %

We now consider the case $2<p<q \leq \infty$. 
%Let us remark that the following upper bound in the range $ c_{p,q} N^2 \le n \le N^{2}-N_q + 1$ is only nontrivial provided that $ n \ge  N^2 - c_q^{-2} \, N^{1+2/p} +1 $.
%We observe that whenever $q$ is much larger than $p$ the second bound is nontrivial for a large part of the intermediate range. 
Observe that in the following bounds the intermediate range is smaller than before, but whenever $q$ is much larger than $p$ it still covers a large part of the former intermediate range $c_{p,q} N^2 \le n \le N^2 - c N^{1+2/q}+1$.  

\begin{proposition}\label{prop:upper bound for 2 leq p leq q leq infty}
Let $2< p < q \leq \infty$ and assume that $n,N\in\N$ with $1\leq n \leq N^2$. Then
\[
c_n\big(\Sc_p^N\hookrightarrow \Sc_q^N\big) \lesssim_q
\begin{cases}
1 & :\, \ 1 \le n \le N^2 - c_q^{-2}N^{1+2/p}+1 \\
N^{1/2-1/p} \, \Big( \frac{N^2-n+1}{N^2} \Big)^{1/2} & :\, \  N^2 - c_q^{-2}N^{1+2/p}+1 \le n \le N^{2}-N^{2/q+1} + 1 \\
N^{1/q-1/p} & :\, \  N^2 - N^{2/q+1}+1 \le n \leq N^2,
\end{cases}
\]
where 
%$c\in(0,1)$ is the absolute constant from Lemma \ref{lem:M} and 
$c_{q} \in (0,\infty)$ depends only on $q$.
\end{proposition}
\begin{proof}
\noindent\emph{Step 1.} First, we consider the case $1 \le n \le  N^2 - c_q^{-2}N^{1+2/p}+1$. We observe that for all $n\in\N$ 
%(and thus particularly for $1 \le n \le c_{p,q} N^2$),
\[
c_n\big(\Sc_p^N\hookrightarrow \Sc_q^N\big) \leq \|\Sc_p^N\hookrightarrow \Sc_q^N \| \leq 1,
\]
because $p<q$. In particular, in this range of small codimension, the constant in the upper bound is simply one (and so independent of $p$ and $q$).
\vskip 1mm
\noindent\emph{Step 2.} Now we look at $  N^2 - c_q^{-2}N^{1+2/p}+1 \le n \le N^{2}$.
In this case we use factorization coupled with the asymptotically sharp estimates from \cite{carldefant1997}.
Using \eqref{eq:CD2}, we obtain
\[
c_n\big(\Sc_p^N\hookrightarrow \Sc_q^N\big) \leq \|\Sc_p^N\hookrightarrow \Sc_2^N \| \, c_n\big(\Sc_2^N\hookrightarrow \Sc_q^N\big) \lesssim_q N^{1/2-1/p} \, \max\Bigg\{ N^{1/q-1/2},\Big( \frac{N^2-n+1}{N^2} \Big)^{1/2}\Bigg\}.
\]
Whenever $n \le N^2 - N^{2/q+1} +1$ the maximum in the previous bound is attained by the second entry.
Therefore,
\[
c_n\big(\Sc_p^N\hookrightarrow \Sc_q^N\big) \le c_q N^{1/2-1/p} \, \Big( \frac{N^2-n+1}{N^2} \Big)^{1/2}
\]
with a constant $c_q \in (0,\infty)$ depending only on $q$. To improve upon the trivial upper bound 1 we need that $ N^2-n+1 \le c_q^{-2} \, N^{1+2/p} $ which is equivalent to $ n \ge  N^2 - c_q^{-2} \, N^{1+2/p} +1 $. In the case $n > N^2 - N^{2/q+1} +1$ the maximum is attained by the first entry, thus giving the upper bound $N^{1/q-1/p}$.
%To improve upon the trivial upper bound 1 we need that $ N^2-n+1 \le c_q^{-2} \, N^{1+2/p} $ which is equivalent to $ n \ge  N^2 - c_q^{-2} \, N^{1+2/p} +1 $.
%\vskip 1mm \textcolor{red}{(J: Do we really need this third step? Doesn't the result follow from factorization and Carl-Defant, since the maximum in Step 2 is attained by the second entry for $n \leq N^2- N^{1+2/q}+1$ and by the first entry for $n \geq N^2- N^{1+2/q}+1$, which also gives $N^{1/q-1/p}$.)}
%\noindent\emph{Step 3.} Finally, let $N^2 -N_q+1 \leq n \leq N^2$. The argument is the same as in the proof of Proposition \ref{prop:upper bound for 1 leq p leq 2 leq q leq infty} Step 1. Applying Lemma \ref{lem:M} to $k:=N^2-n+1 (\leq N_q = cN^{2/q+1})$ yields a subspace $L\subset \Sc_q^N$ with $\dim L \geq k$ such that, for all $A\in L$,
%\[
%c_1(q)^{-1} N^{1/2-1/q} \|A\|_{\Sc_q} \leq \|A\|_{\Sc_2} \leq c_2 N^{-1/2}\|A\|_{\Sc_1}.
%\]
%As in Step 1 of the proof of Proposition \ref{prop:upper bound for 1 leq p leq 2 leq q leq infty}, we get that, for all $A\in L$,
%\begin{equation}\label{eq:q-norm upper bound by p-norm on L}
% \|A\|_{\Sc_q} 
%   \le c_1(q)c_2 N^{1/q-1/p} \|A\|_{\Sc_p} .
%\end{equation}
%Now, since $\codim L <n $, algebraically considering $L$ as a subspace of $\Sc_p^N$, we obtain from \eqref{eq:q-norm upper bound by p-norm on L} that 
%\[
%c_n\big(\Sc_p^N\hookrightarrow \Sc_q^N\big) \leq \sup_{ \substack{ A\in L \\ {\|A\|_{\Sc_p}}\leq 1}}\|A\|_{\Sc_q} \leq c_1(q)c_2 N^{1/q-1/p}.
%\]
%In particular, we see that the constant in the upper bound does not depend on the parameter $p$.
\end{proof}

\begin{rmk}
 Comparing the previous upper bounds with the lower bounds obtained in \cite{HM2005}, we see that  in the ranges of small and large codimension $n$,
\[
c_n\big(\Sc_p^N\hookrightarrow \Sc_q^N\big) \asymp_{p,q} 
\begin{cases}
1 & :\, \ 1 \le n \le c_{p,q} N^2 \\
N^{1/q-1/p} & :\, \  N^2 - N_q+1 \le n \leq N^2,
\end{cases}
\]
where $N_q:= c N^{2/q+1}$ denotes the critical dimension and the constant $c \in (0,1)$ is the constant from Lemma \ref{lem:M}.
This is exactly what is stated in Theorem \ref{thm:main sharp bounds} in the last two cases.
In the intermediate range $c_{p,q} N^2 \le n \le N^{2}-N_q + 1 $ we recover the part $\Big( \frac{N^2-n+1}{N^2} \Big)^{1/2}$ without the exponent that arises from an interpolation argument, but with an additional factor $N^{1/2-1/p}$ coming from the factorization argument.
We believe the lower bound to be asymptotically sharp.
For $p$ tending to 2 we observe that the bounds blend into each other since both the factorization part and the interpolation exponents vanish.
\end{rmk}

% % % % % % % % % % % % % % % % % % %
\subsection{The upper bound in the case $1<p<q\le 2$}
% % % % % % % % % % % % % % % % % % %

We now consider the case where $1<p<q\le 2$ and show that, contrary to what was claimed in \cite{CDK2015}, the bound for $0<p\leq 1$ and $p<q\leq 2$ does not carry over to this regime.
The proof of our estimate is based on the interpolation Lemma \ref{lem:gluskin interpolation lemma} in combination with the first part of
Proposition \ref{prop:upper bound for 1 leq p leq 2 leq q leq infty}. More precisely, in the special case where $q=2$,
Proposition \ref{prop:upper bound for 1 leq p leq 2 leq q leq infty} (or simply \eqref{eq:CD}) states that
\begin{equation}\label{eq:1pq2}
c_n\big(\Sc_p^N\hookrightarrow \Sc_2^N\big)\lesssim_p \min\Biggl\{1,\frac{N^{3/2-1/p}}{n^{1/2}}\Biggr\}
\end{equation}
for all $1\le n\le N^2$. %\textcolor{red}{see also \eqref{eq:CD}}.
In combination with Lemma \ref{lem:gluskin interpolation lemma}, which we use to interpolate between $\Sc_p^N$ and $\Sc_2^N$,
we obtain the following upper bound on the Gelfand numbers.
 
\begin{proposition} Let $1\le p\le q\le 2$ and assume that $n,N\in\N$ with $1\leq n \leq N^2$. Then
\begin{equation*}
c_n\big(\Sc_p^N\hookrightarrow \Sc_q^N\big)\lesssim_{p} \min\Biggl\{1,\frac{N^{3/2-1/p}}{n^{1/2}}\Biggr\}^{\frac{1/p-1/q}{1/p-1/2}}.
\end{equation*}
%\textcolor{red}{(J: \"Ubersehe ich hier etwas? Ich habe den Beweis jetzt genauer aufgeschrieben und erhalte, dass die Konstante in der Ungleichung nur von $p$ abh\"angt, oder?)}
\end{proposition}
\begin{proof} 
Since the Schatten classes satisfy a H\"older-type inequality, condition \eqref{eq:hoelder assumption interpolation} is satisfied
for the choice $X_1:= \Sc_p^N$, and $X_0:=\Sc_2^N$, where $p < q \leq 2$. In particular, we can find $\theta\in(0,1)$ 
%as in \eqref{eq: choice theta for q}, i.e., 
such that
\begin{equation*} %\label{eq: choice theta for q_II}
\frac{1}{q} = \frac{1-\theta}{2} + \frac{\theta}{p}. 
\end{equation*}
This choice means that
\begin{equation*} %\label{eq:choiceq2}
1-\theta = \frac{\frac{1}{p} - \frac{1}{q}}{\frac{1}{p}-\frac{1}{2}}.
\end{equation*}
Applying Lemma \ref{lem:gluskin interpolation lemma} with $X_\theta:=\Sc_q^N$, and combining this with \eqref{eq:1pq2}, we obtain
\[
c_n\big(\Sc_p^N\hookrightarrow \Sc_q^N\big) \leq c_n\big(\Sc_p^N\hookrightarrow \Sc_2^N\big)^{1-\theta}\lesssim_p \min\Biggl\{1,\frac{N^{3/2-1/p}}{n^{1/2}}\Biggr\}^{\frac{1/p - 1/q}{1/p-1/2}},
\]
which completes the proof.
\end{proof}

% % % % % % % % % % % % % % % % % % % % % % % % % % % % % % % % % % % % % % % % %
\subsection{The upper bound in the case $0< p \leq 1$ and $ 2\leq q \leq \infty$}
% % % % % % % % % % % % % % % % % % % % % % % % % % % % % % % % % % % % % % % % %

In this subsection we provide upper bounds for the case $0< p \leq 1$ and $ 2\leq q \leq \infty$, where the domain space is a quasi-Banach space. While for large codimensions $n$, i.e., $N^2 - N_q+1 \le n \leq N^2$, the bound is indeed asymptotically sharp, the other bounds do not match their lower counterparts unless $q=2$.

\begin{proposition}\label{prop:upper bound for 0 leq p leq 1 and 2 leq q leq infty}
Let $0< p \leq 1 $ and $2 \leq q \leq \infty$ and assume that $n,N\in\N$ with $1\leq n \leq N^2$. Then
\[
c_n\big(\Sc_p^N\hookrightarrow \Sc_q^N\big) \lesssim_{p,q} 
\begin{cases}
\min\Big\{1,\frac{N}{n} \Big\}^{1/p-1/2} & :\, \ 1 \le n \le (1-c) N^2 \\
N^{-1/p-1/2} (N^2 - n +1)^{1/2} & :\, \  (1-c) N^2 \le n \le N^{2}-N_q+1 \\
N^{1/q-1/p} & :\, \  N^2 - N_q+1 \le n \leq N^2,
\end{cases}
\]
where $N_q:= c N^{2/q+1}$ denotes the critical dimension and the constant $c \in (0,1)$ is the constant from Lemma \ref{lem:M}.
\end{proposition}
\begin{proof} 
\emph{Step 1.}
Let $1\leq n \leq (1-c)N^2$. Then 
\[
c_n\big(\Sc_p^N\hookrightarrow \Sc_q^N\big) \leq c_n\big(\Sc_p^N\hookrightarrow \Sc_2^N\big) \underbrace{\|\Sc_2^N\hookrightarrow \Sc_q^N \|}_{=1} \lesssim_{p} \min\Big\{1,\frac{N}{n} \Big\}^{1/p-1/2},
\]
where we used \eqref{eq:cdk} with $q=2$. So the constant in this bound depends on $p$ but not on $q$.

\vskip 1mm
\noindent\emph{Step 2.}
For the remaining two ranges $N^2 - N_q+1 \le n \leq N^2$ and $(1-c) N^2 \le n \le N^{2}-N_q+1$, the proof follows (in that order)
the argument for the corresponding ranges in the proof of Proposition \ref{prop:upper bound for 1 leq p leq 2 leq q leq infty}.
First, if $L$ is the subspace from Lemma \ref{lem:M}, we combine \eqref{eq:Dvoretzky0}
with the interpolation inequality
\begin{equation*}
\|A\|_{\Sc_1}\le \|A\|_{\Sc_p}^{1-\theta}\cdot\|A\|_{\Sc_q}^{\theta}\quad\text{ for }\theta\text{ given by }\quad 1=\frac{1-\theta}{p}+\frac{\theta}{q}
\end{equation*}
and obtain
\[
c_1(q)^{-1}N^{1/2-1/q}\|A\|_{\Sc_q}\le \|A\|_{\Sc_2}\le c_2 N^{-1/2}\|A\|_{\Sc_1}\le [c_1(q)c_2]^{\frac{1-1/p}{1/q-1}} c_2 N^{1/2-1/p}\|A\|_{\Sc_p}.
\]
After a simple calculation, this becomes (we may assume that $c_1(q)c_2\ge 1$)
\[
\|A\|_{\Sc_q}\le [c_1(q)c_2]^{1+2(1/p-1)}N^{1/q-1/p}\|A\|_{\Sc_p},
\]
which gives the proof in the range $N^2 - N_q+1 \le n \leq N^2$ with constants depending on both $p$ and $q$.
For $(1-c) N^2 \le n \le N^{2}-N_q+1$, the proof is then the same as that for Proposition \ref{prop:upper bound for 1 leq p leq 2 leq q leq infty}.
\end{proof}

\begin{rmk}
First of all, a simple computation shows that the upper bound for small codimensions $1 \le n \le (1-c) N^2$ (which is a bound that obviously holds for any $1\leq n \leq N^2$) is indeed weaker than the one we wrote for $(1-c) N^2 \le n \le N^{2}-N_q+1$. We omit the details.

Comparing the upper bound for $1 \le n \le (1-c) N^2$ in Proposition \ref{prop:upper bound for 0 leq p leq 1 and 2 leq q leq infty} with the lower one from \eqref{eq:cdk}, shows that
\[
\min\Big\{1,\frac{N}{n} \Big\}^{1/p-1/q}\lesssim_{p,q} c_n\big(\Sc_p^N\hookrightarrow \Sc_q^N\big) \lesssim_p \min\Big\{1,\frac{N}{n} \Big\}^{1/p-1/2},
\] 
which is only sharp for small codimension $n$ or when $q=2$. For $(1-c) N^2 \le n \le N^{2}-N_q+1$, comparing again with \eqref{eq:cdk}, we see that
\[
\min\Big\{1,\frac{N}{n} \Big\}^{1/p-1/q}\lesssim_{p,q} c_n\big(\Sc_p^N\hookrightarrow \Sc_q^N\big) \lesssim_{q} N^{-1/p-1/2} (N^2 - n +1)^{1/2}.
\] 
The bound in the case of large codimension $N^2 - N_q+1 \le n \leq N^2$ is sharp (up to constants depending on $p$), because 
\[
c_n\big(\Sc_p^N\hookrightarrow \Sc_q^N\big) \geq c_{N^2}\big(\Sc_p^N\hookrightarrow \Sc_q^N\big) = \frac{1}{\|\Sc_q^N\hookrightarrow \Sc_p^N\|} = \frac{1}{N^{1/p-1/q}} = N^{1/q-1/p}.
\]
\end{rmk}

% % % % % % % % % % % % % % % % % % % % % % % % % % % % % %
\subsection{The lower bound in the case $0< p < q \leq 2$}
% % % % % % % % % % % % % % % % % % % % % % % % % % % % % %

Ch\'avez-Dom\'inguez and Kutzarova \cite{CDK2015}, following essentially the technique of \cite{FPRU2010}, showed that
\begin{equation}\label{eq:cdk bound for 1 leq p leq q leq 2}
c_n\big(\Sc_p^N\hookrightarrow \Sc_q^N\big) \asymp_{p,q} \min\bigg\{ 1,\frac{N}{n} \bigg\}^{1/p-1/q}
\end{equation}
for $0<p\leq 1$ and $p<q\leq 2$, the lower bound carrying over to the case $q>2$.
Using Carl's inequality for quasi-Banach spaces \cite{HKV2016} and bounds on entropy numbers of natural embeddings between
Schatten classes, another and quite short proof of the lower bound in this regime was given in \cite{HPV2017}. Let us remark that, while the authors in \cite{CDK2015} claim that their (upper) bound carries over to the case $1< p <q \leq 2$ via a simple interpolation argument, this is in fact not true as our results will show.
% Let us remark here that the upper bound, but also for $1\leq p < q \leq 2$.
%This can be seen by following verbatim the proof of Theorem 5.3 in \cite{CDK2015} where the choice $r=\min\{1,q\}$ must be replaced by a choice $r\in(p,q)$.
%However, this still leaves open the lower bound in this regime, which is covered by the following proposition. 

Since the case $0<p\le 1$ is settled by \cite{CDK2015} and \eqref{eq:cdk bound for 1 leq p leq q leq 2}, we restrict ourselves to $1\le p\le q\le 2$.
%\begin{proposition} Let $1\le p \le q\le 2$ and assume that $n,N\in\N$ with $1\le n\le N^2$. Then
%\begin{equation}\label{eq:1pq2below}
%c_n\big(\Sc_p^N\hookrightarrow \Sc_q^N\big) \gtrsim_{p} \min\left\{1,\Bigl(\frac{N^{3/2-1/p}}{n^{1/2}}\Bigr)^{\frac{1/p-1/q}{1/p-1/2}}\right\}.
%\end{equation}
%\textcolor{red}{(J: We should write the bound for $C_{p}N^2\le n\le N^2$ - see Step 2 - as well, right?! Maybe it's best to write all three cases for the lower bound, as we did in other propositions as well.)}
%\end{proposition}
\begin{proposition} Let $1\le p \le q\le 2$ and assume that $n,N\in\N$ with $1\le n\le N^2$. Then
\begin{equation}\label{eq:1pq2below}
c_n\big(\Sc_p^N\hookrightarrow \Sc_q^N\big) \gtrsim_{p} \begin{cases}
1 & :\, \ 1 \le n \le c_p N^{3-2/p} \\
\Bigl(\frac{N^{3/2-1/p}}{n^{1/2}}\Bigr)^{\frac{1/p-1/q}{1/p-1/2}} & :\, \  c_p N^{3-2/p} \le n \le C_p N^{2} \\
N^{1/q-1/p} & :\, \  C_p N^2 \le n \leq N^2,
\end{cases}.
\end{equation}
\end{proposition}
\begin{rmk}
 We observe that the bound in the preceding proposition can be written more compactly, but less clearly, as
\[
c_n\big(\Sc_p^N\hookrightarrow \Sc_q^N\big) \gtrsim_{p} \min\left\{1,\Bigl(\frac{N^{3/2-1/p}}{n^{1/2}}\Bigr)^{\frac{1/p-1/q}{1/p-1/2}}\right\}.
\]
This is the form used in the formulation of Theorem \ref{thm:main sharp bounds}.
\end{rmk}
\begin{proof} 
\emph{Step 1.} Consider the case $1\le n\le c_{p} N^{3-2/p}$. We exploit the connection between Schatten norms and mixed norms of matrices (see \eqref{eq:pq1} and \eqref{eq:pq2}) and the lower bounds
on Kolmogorov numbers of embeddings of mixed Lebesgue spaces of Vasil'eva \cite{Vasil13}. Inspired by \cite{G81}, Vasil'eva defined
the sets
$$
V_{1,1}^{N,N}:=\conv\big\{\pm e_{i,j}\,:\,i,j=1,\dots,N\big\}=B_{\ell_1^N(\ell_1^N)}\subset \R^{N\times N},
$$
where $e_{i,j}\in\R^{N\times N}$, $i,j\in\{1,\dots,N\}$ are the $N\times N$ matrices with one entry in the $i^{\rm th}$ row and $j^{\rm th}$ column equal to one
and the other entries equal to zero. The Kolmogorov widths of $V_{1,1}^{N,N}$ in the mixed norm spaces were estimated in formula (34) of \cite{Vasil13}, where the author obtained the lower bound 
\begin{align}\label{eq:lower bound vasileva}
d_n\Bigl(V_{1,1}^{N,N},\ell_{p^{*}}^N(\ell_2^N)\Bigr)& \gtrsim_{p^{*}} 1,
\end{align}
on the Kolmogorov width of $V_{1,1}^{N,N}$ in $\ell_{p^{*}}^N(\ell_2^N)$, whenever
\[
1\le n \lesssim_{p^{*}} N^{1+2/p^{*}}=N^{3-2/p}.
\]
Here, $p^{*}$ denotes again the H\"older conjugate of $p$.
Using duality and
\eqref{eq:pq1} together with \eqref{eq:lower bound vasileva}, we obtain
$$
c_n\big(\Sc_p^N\hookrightarrow \Sc_q^N\big)=d_n\big(\Sc^N_{q^{*}}\hookrightarrow \Sc^N_{p^{*}}\big)\ge d_n\big(V_{1,1}^{N,N},\Sc_{p^{*}}^N\big)\ge d_n\Bigl(V_{1,1}^{N,N},\ell_{p^{*}}^N(\ell_2^N)\Bigr)
\gtrsim_{p^{*}}1
$$
for $1\leq n\le c_p N^{1+2/p^{*}}=c_pN^{3-2/p}$ and a constant $c_p\in(0,\infty)$ only depending on $p$.
\vskip 1mm
\noindent\emph{Step 2.} Now we consider the range $C_{p}N^2\le n\le N^2$, where $C_{p}\in(0,\infty)$ will be determined by Step 3. In this case (actually for all $1\leq n \leq N^2$), the lower bound follows from the fact that Gelfand numbers are decreasing in $n$, which yields
$$
c_n\big(\Sc_p^N\hookrightarrow \Sc_q^N\big)\ge c_{N^2}\big(\Sc_p^N\hookrightarrow \Sc_q^N\big)= \| \Sc_q^N\hookrightarrow \Sc_p^N \|^{-1} =N^{1/q-1/p}.
$$
\vskip 1mm
\noindent\emph{Step 3.} For the case $c_{p}N^{3-2/p}\le n\le C_{p}N^2$, we adapt the technique of Gluskin \cite{G81} and Vasil'eva \cite{Vasil13}. For this sake, let $r\in\N$ be such that $1\le r\le N$
and let $A^r\in\R^{N\times N}$ be the $N\times N$ matrix with $(A^r)_{i,j}=1$ if $1\le i=j\le r$ and all the other coordinates equal to zero.
Hence, $A^r$ is a diagonal matrix with its first $r$ entries on the diagonal equal to one and the others equal to zero. Therefore, $\|A^r\|_{\Sc_t}=r^{1/t}$ for all $1\le t\le \infty.$
We denote by $G$ the set
$$
G:=\big\{(\pi_1,\pi_2,\varepsilon): \pi_1,\pi_2\in\Pi_N,\varepsilon\in\{-1,+1\}^N\big\},
$$
where $\Pi_N$ denotes the symmetric group of permutations on the set $\{1,\dots,N\}$.
For $\gamma=(\pi_1,\pi_2,\varepsilon)\in G$, we define $\gamma(A^r):=(\varepsilon_i A^r_{\pi_1(i),\pi_2(j)})_{1\le i,j\le N}$
and introduce the following averaged set of matrices 
\[
  {\mathcal V}_r^N:=\big\{\gamma(A^r)\,:\,\gamma\in G\big\}.
\]

We will show later (in Step 4.) the following generalization of \eqref{eq:lower bound vasileva}, which gives
\begin{equation}\label{eq:GluVas1}
d_n\big({\mathcal V}_r^N,\ell_{p^{*}}^N(\ell_2^N)\big)\ge c_{p^{*}} r^{1/p^{*}}\quad\text{for}\quad 1\leq n\le C_{p^{*}}N^{1+2/p^{*}}r^{1-2/p^{*}},
\end{equation}
with constants $c_{p^{*}}, C_{p^{*}}\in(0,\infty)$ depending only on $p^{*}$.
Assuming \eqref{eq:GluVas1} and using in this order duality, the fact that $\|\gamma(A^{r})\|_{q^{*}}=r^{1/q^{*}}$, and \eqref{eq:pq1} followed by \eqref{eq:GluVas1}, we obtain for all $1\leq n\le C_{p^{*}}N^{1+2/p^{*}}r^{1-2/p^{*}}$ that
\begin{align*}
c_n\big(\Sc_p^N\hookrightarrow \Sc_q^N\big)&=d_n\big(\Sc_{q^{*}}^N\hookrightarrow \Sc_{p^{*}}^N\big)\ge r^{-1/q^{*}}d_n\big({\mathcal V}_r^N,\Sc_{p^{*}}^N\big)\\
&\ge r^{-1/q^{*}}d_n\big({\mathcal V}_r^N,\ell_{p^{*}}^N(\ell_2^N)\big)\ge c_{p^{*}}r^{1/p^{*}-1/q^{*}}=c_{p^{*}}r^{1/q-1/p}.
\end{align*}
Now we need to choose a suitable $r$. We take 
$$
r:=\left\lceil\Bigl(C_{p^{*}}^{-1}nN^{-1-2/p^{*}}\Bigr)^{\frac{1}{1-2/p^{*}}}\right\rceil=\left\lceil\Bigl(C_{p^{*}}^{-1}nN^{-3+2/p}\Bigr)^{\frac{1}{2/p-1}}\right\rceil
$$
and observe that indeed $1\le r\le N$ when $n\in\N$ is such that $c_{p}N^{3-2/p}\le n\le C_{p}N^2$ for a sufficiently small constant $C_{p}\in(0,\infty)$, more precisely, whenever $C_p \leq C_{p^*}$.
This choice then leads to
$$
c_n(\Sc_p^N\hookrightarrow \Sc_q^N)\gtrsim_{p}\Bigl(nN^{-3+2/p}\Bigr)^{\frac{1/q-1/p}{2/p-1}}=\Bigl(n^{1/2}N^{-3/2+1/p}\Bigr)^{\frac{1/q-1/p}{1/p-1/2}}.
$$
\vskip 1mm
\noindent\emph{Step 4.}
It remains to prove \eqref{eq:GluVas1}. In order to do this, we follow \cite{Vasil13}. Let $Y\subset \R^{N\times N}$ be a subspace of dimension at most $n$.
For $\gamma\in G$, we denote by $y^\gamma=(y^\gamma_{i,j})_{1\le i,j\le N}$ the nearest element of $Y$ to $\gamma(A^r)$ in the $\ell_{p^{*}}^N(\ell_2^N)$ norm.
Furthermore, we denote
$$
I_1^\gamma=\big\{j\in\{1,\dots,N\}:\gamma(A^r)_{i,j}=0\ \text{for all}\ i=1,\dots,N\big\},\quad I_2^\gamma=\{1,\dots,N\}\setminus I_1^\gamma
$$
and, for $\gamma\in G$ and $j\in I_2^\gamma$,
$$
J_{1,j}^\gamma=\big\{i\in\{1,\dots,N\}:\gamma(A^r)_{i,j}=0\big\}\ \text{and}\ J_{2,j}^\gamma=\{1,\dots,N\}\setminus J_{1,j}^\gamma.
$$
Observe that by the construction of $A^r$ and $G$, $J_{2,j}^\gamma$ are singletons.

Using this notation, we can estimate
\begin{align*}
\sum_{j=1}^N&\Bigl(\sum_{i=1}^N |\gamma(A^r)_{i,j}-y^\gamma_{i,j}|^{2} \Bigr)^{p^{*}/2}=
\sum_{j\in I_1^\gamma}\Bigl(\sum_{i=1}^N |y^\gamma_{i,j}|^{2} \Bigr)^{p^{*}/2}+
\sum_{j\in I_2^\gamma}\Bigl(\sum_{i\in J^\gamma_{1,j}} |y^\gamma_{i,j}|^{2}+\sum_{i\in J^\gamma_{2,j}} |\gamma(A^r)_{i,j}-y^\gamma_{i,j}|^{2} \Bigr)^{p^{*}/2}\\
&=\sum_{j\in I_1^\gamma}\Bigl(\sum_{i=1}^N |y^\gamma_{i,j}|^{2} \Bigr)^{p^{*}/2}+
\sum_{j\in I_2^\gamma}\Bigl(\sum_{i\in J^\gamma_{1,j}} |y^\gamma_{i,j}|^{2}+\sum_{i\in J^\gamma_{2,j}} |1-\gamma(A^r)_{i,j}y^\gamma_{i,j}|^{2} \Bigr)^{p^{*}/2}\\
&\ge \sum_{j\in I_1^\gamma}\Bigl(\sum_{i=1}^N |y^\gamma_{i,j}|^{2} \Bigr)^{p^{*}/2}
+\frac{1}{2}\sum_{j\in I_2^\gamma}\Biggl(\sum_{i\in J^\gamma_{1,j}} |y^\gamma_{i,j}|^{2}\Biggr)^{p^{*}/2}
+\frac{1}{2}\sum_{j\in I_2^\gamma}\Biggl(\sum_{i\in J^\gamma_{2,j}} |1-\gamma(A^r)_{i,j}y^\gamma_{i,j}|^{2}\Biggr)^{p^{*}/2},
\end{align*}
where we used $(a+b)^\theta\ge (a^\theta+b^\theta)/2$ for $a,b,\theta>0$.
Using \cite[Proposition 1]{Vasil13} , the last summand can be further estimated as
\begin{align*}
 \frac{1}{2}\sum_{j\in I_2^\gamma}\Biggl(\sum_{i\in J^\gamma_{2,j}} |1-\gamma(A^r)_{i,j}y^\gamma_{i,j}|^{2}\Biggr)^{p^{*}/2}
 &\ge \frac{1}{2}\sum_{j\in I_2^\gamma}\Biggl(\frac{1}{2}+c_1(p^{*}) \Biggl(\sum_{i\in J^\gamma_{2,j}} |y^\gamma_{i,j}|^{2}\Biggr)^{p^{*}/2}-p^{*}\sum_{i\in J^\gamma_{2,j}}\gamma(A^r)_{i,j}y^\gamma_{i,j} \Biggr)\\
 &=\frac{r}{4}+
\frac{c_1(p^{*})}{2}\sum_{j\in I_2^\gamma}\Biggl(\sum_{i\in J^\gamma_{2,j}} |y^\gamma_{i,j}|^{2}\Biggr)^{p^{*}/2}
-\frac{p^{*}}{2}\sum_{j\in I_2^\gamma}\sum_{i\in J^\gamma_{2,j}}\gamma(A^r)_{i,j}y^\gamma_{i,j}\\
 &=\frac{r}{4}+
\frac{c_1(p^{*})}{2}\sum_{j\in I_2^\gamma}\Biggl(\sum_{i\in J^\gamma_{2,j}} |y^\gamma_{i,j}|^{2}\Biggr)^{p^{*}/2}
-\frac{p^{*}}{2}\sum_{j=1}^N\sum_{i=1}^N\gamma(A^r)_{i,j}y^\gamma_{i,j}.
\end{align*}
Altogether, letting $c_1'(p^{*}) = \min \{ 1/2,c_1(p^{*})/2 \}$, we arrive at 
\begin{align*}
\sum_{j=1}^N\Bigl(\sum_{i=1}^N |\gamma(A^r)_{i,j}-y^\gamma_{i,j}|^{2} \Bigr)^{p^{*}/2}
&\ge \frac{r}{4}+c_1'(p^{*})\sum_{j=1}^N\Bigl(\sum_{i=1}^N |y^\gamma_{i,j}|^{2} \Bigr)^{p^{*}/2}
-\frac{p^{*}}{2}\sum_{j=1}^N\sum_{i=1}^N\gamma(A^r)_{i,j}y^\gamma_{i,j}.
\end{align*} 

Averaging over $\gamma\in G$, we obtain
\begin{align}\label{eq:VasilMax1}
\max_{\gamma\in G}\|\gamma(A^r)-y^\gamma\|_{\ell_{p^{*}}(\ell_2)}^{p^{*}}&\ge |G|^{-1}\sum_{\gamma\in G} \|\gamma(A^r)-y^\gamma\|_{\ell_{p^{*}}(\ell_2)}^{p^{*}}
=|G|^{-1}\sum_{\gamma\in G}\sum_{j=1}^N\Bigl(\sum_{i=1}^N |\gamma(A^r)_{i,j}-y^\gamma_{i,j}|^{2} \Bigr)^{p^{*}/2}\\
\notag&\ge \frac{r}{4}+c_1'(p^{*})|G|^{-1}\sum_{\gamma\in G}\sum_{j=1}^N\Bigl(\sum_{i=1}^N |y^\gamma_{i,j}|^{2} \Bigr)^{p^{*}/2}
-\frac{p^{*}}{2}|G|^{-1}\sum_{\gamma\in G}\sum_{j=1}^N\sum_{i=1}^N\gamma(A^r)_{i,j}y^\gamma_{i,j}.
\end{align}

The next step is to estimate the absolute value of the last term. This is where the dimension of $Y$ comes into play.
We consider the space $\ell_2(G)=\{\varphi:G\to \R\}$ equipped with the inner product
$$
\langle \varphi,\psi\rangle=|G|^{-1}\sum_{\gamma\in G}\varphi(\gamma)\psi(\gamma).
$$
For $1\le i,j\le N$, we define $\varphi_{i,j},z_{i,j}\in\ell_2(G)$ by
$$
\varphi_{i,j}(\gamma)=\gamma(A^r)_{i,j}\quad\text{and}\quad z_{i,j}(\gamma)=y^\gamma_{i,j}.
$$
Let $L:={\rm span}\{z_{i,j}:1\le i,j\le N\}$.
Then we claim that $\dim L\le n$. Indeed, if we arrange the vectors $y^\gamma=(y^\gamma_{i,j})_{1\le i,j\le N}$
%indexed by the pairs $(i,j)\in\{1,\dots,N\}^2$
as rows of a matrix, then the vectors $z_{i,j}=(z_{i,j}(\gamma))_{\gamma\in G}$
are the columns of this matrix, $L$ is the linear span of its columns, and $\dim L$ is its rank. But, by the construction,
$y^\gamma\in Y$ for every $\gamma\in G$ and $\dim Y\le n$. Therefore, also $\dim L\le n$.
Let $P$ be the orthogonal projector onto $L$. Let us recall
that its Hilbert-Schmidt norm is at most $n^{1/2}.$

Next we observe that $\{\varphi_{i,j}:1\le i,j\le N\}$ forms an orthogonal system in $\ell_2(G)$ with $\|\varphi_{i,j}\|^2_{\ell_2(G)}=\frac{r}{N^2}$
for every $1\le i,j\le N$, i.e., 
\begin{equation}\label{eq:Vasil1}
|G|^{-1}\sum_{\gamma\in G}\varphi_{i,j}(\gamma)\varphi_{i',j'}(\gamma)=|G|^{-1}\sum_{\gamma\in G}\gamma(A^r)_{i,j}\gamma(A^r)_{i',j'}
=\begin{cases}\frac{r}{N^2}&:\,\ i=i'\ \text{and} j=j',\\0&\text{otherwise}.
\end{cases}
\end{equation}
Indeed, if $i=i'$ and $j=j'$, then we use that $|G|=(N!)^2\cdot 2^N$ and $\gamma(A^r)_{i,j}=\varepsilon_iA^r_{\pi_1(i),\pi_2(j)}$. Therefore,
\begin{align*}
|G|^{-1}\sum_{\gamma\in G}|\gamma(A^r)_{i,j}|^2&=\frac{1}{(N!)^2\cdot 2^N}\sum_{\varepsilon\in\{-1,+1\}^N}\sum_{\pi_1\in\Pi_N}\sum_{\pi_2\in\Pi_N}|\gamma(A^r)_{i,j}|^2\\
&=\frac{1}{(N!)^2}\sum_{\pi_1\in\Pi_N}\sum_{\pi_2\in\Pi_N}A^r_{\pi_1(i),\pi_2(j)}=\frac{|\{(\pi_1,\pi_2)\in\Pi_N\times\Pi_N:\pi_1(i)=\pi_2(j)\le r\}|}{(N!)^2}\\
&=\frac{r\cdot [(N-1)!]^2}{(N!)^2}=\frac{r}{N^2}.
\end{align*}
If $i=i'$ and $j\not=j'$, then $\pi_2(j)\not=\pi_2(j')$ and $A^r_{\pi_1(i),\pi_2(j)}A^r_{\pi_1(i),\pi_2(j')}=0.$ Similarly, 
if $i\not=i'$ and $j=j'$, then $\pi_1(i)\not=\pi_1(i')$ and again $A^r_{\pi_1(i),\pi_2(j)}A^r_{\pi_1(i'),\pi_2(j)}=0.$
Finally, if $i\not=i'$ and $j\not=j'$, then
$$
\sum_{\varepsilon\in\{-1,+1\}^N}\varepsilon_i \varepsilon_{i'}A^r_{\pi_1(i),\pi_2(j)}A^r_{\pi_1(i'),\pi_2(j')}=0
$$
and \eqref{eq:Vasil1} follows.

This allows us to continue in the estimate of the absolute value of the last term in \eqref{eq:VasilMax1}
\begin{align}\notag
\biggl| |G|^{-1}\sum_{\gamma\in G}&\sum_{j=1}^N\sum_{i=1}^N\gamma(A^r)_{i,j}y^\gamma_{i,j}\biggr|=
\biggl||G|^{-1}\sum_{\gamma\in G}\sum_{j=1}^N\sum_{i=1}^N\varphi_{i,j}(\gamma)z_{i,j}(\gamma)\biggr|=
\biggl|\sum_{j=1}^N\sum_{i=1}^N\langle \varphi_{i,j},z_{i,j}\rangle\biggr|\\
\label{eq:VasilMax2}&=\biggl|\sum_{j=1}^N\sum_{i=1}^N\langle P\varphi_{i,j},z_{i,j}\rangle\biggr|\le \sum_{j=1}^N\sum_{i=1}^N|\langle P\varphi_{i,j},z_{i,j}\rangle|\\
\notag&\le \sum_{j=1}^N\sum_{i=1}^N\|P\varphi_{i,j}\|_{\ell_2(G)}\cdot \|z_{i,j}\|_{\ell_2(G)}
\le \biggl(\sum_{j=1}^N\sum_{i=1}^N\|P\varphi_{i,j}\|^2_{\ell_2(G)}\biggr)^{1/2}
\biggl(\sum_{j=1}^N\sum_{i=1}^N \|z_{i,j}\|^2_{\ell_2(G)}\biggr)^{1/2}\\
\notag&\le n^{1/2}\cdot\Bigl(\frac{r}{N^2}\Bigr)^{1/2}\biggl(\sum_{j=1}^N\sum_{i=1}^N \|z_{i,j}\|^2_{\ell_2(G)}\biggr)^{1/2}
=n^{1/2}\cdot\Bigl(\frac{r}{N^2}\Bigr)^{1/2}\biggl(|G|^{-1}\sum_{\gamma\in G}\sum_{j=1}^N\sum_{i=1}^N |y^{\gamma}_{i,j}|^2\biggr)^{1/2}\\
\notag&\le n^{1/2}\cdot\Bigl(\frac{r}{N^2}\Bigr)^{1/2}N^{1/2-1/p^{*}}\Biggl(|G|^{-1}\sum_{\gamma\in G}\sum_{j=1}^N\Biggl(\sum_{i=1}^N |y^{\gamma}_{i,j}|^2\Biggr)^{p^{*}/2}\Biggr)^{1/p^{*}}.
\end{align}
We denote 
$$
\Gamma=|G|^{-1}\sum_{\gamma\in G}\sum_{j=1}^N\Biggl(\sum_{i=1}^N |y^{\gamma}_{i,j}|^2\Biggr)^{p^{*}/2}
$$
and combine \eqref{eq:VasilMax1} with \eqref{eq:VasilMax2} and Young's inequality $ab\le a^p/p+b^{p^{*}}/p^{*}\le a^p+b^{p^{*}}$, to further estimate
\begin{align*}
\max\|\gamma(A^r)-y^\gamma\|_{\ell_{p^{*}}(\ell_2)}^{p^{*}}&\ge \frac{r}{4}+c_1'(p^{*})\Gamma-\frac{p^{*}}{2}(nr)^{1/2}N^{-1/2-1/p^{*}}\Gamma^{1/p^{*}}\\
&\ge \frac{r}{4}+c_1'(p^{*})\Gamma-\frac{p^{*}}{2}n^{1/2}N^{-1/2-1/p^{*}}r^{1/p^{*}-1/2}r^{1-1/p^{*}}\Gamma^{1/p^{*}}\\
&\ge \frac{r}{4}+c_1'(p^{*})\Gamma-\frac{p^{*}}{2}n^{1/2}N^{-1/2-1/p^{*}}r^{1/p^{*}-1/2}(r+\Gamma).
\end{align*}
If now $n\le c_{p}N^{1+2/p^{*}}r^{1-2/p^{*}}$ for $c_p\in(0,\infty)$ small enough, we obtain \eqref{eq:GluVas1}.
\end{proof}

\subsection{The lower bound in the case $1\le p\le 2\le q\le \infty$}
% % % % % % % % % % % % % % % % % % % % % % % % % % % % % % % % % % %

We shall now prove a lower bound for the Gelfand numbers in the regime $1<p\le 2\le q\le \infty$ when $1\le n\le c_{p,q}N^2$ and $N^2 - N_q+1 \le n \leq N^2$, where $N_q:= c N^{2/q+1}$ and $c_{p,q}\in(0,1)$ is a constant only depending on $p$ and $q$.
In these ranges, the upper bound of Proposition \ref{prop:upper bound for 1 leq p leq 2 leq q leq infty} and the lower bound of Proposition \ref{prop: lower bound 1 < p <2 <q<infty} match and are therefore optimal.
We leave it an open problem to find a good lower bound also in the intermediate range $c_{p,q}N^2\le n \le N^2-N_q+1$.

\begin{proposition}\label{prop: lower bound 1 < p <2 <q<infty} Let $1\le p\le 2\le q\le \infty$. Then there exists a number $c_{p,q}\in(0,1)$ such that, for all $n,N\in\N$ with $1\leq n \leq N^2$, we have
\[
c_n\big(\Sc_p^N\hookrightarrow \Sc_q^N\big)\gtrsim_{p,q}
\begin{cases}
\min\Big\{1,\frac{N^{3/2-1/p}}{n^{1/2}}\Big\} & :\, \ 1 \le n \le c_{p,q} N^2 \\
N^{1/q-1/p} & :\, \  N^2 - N_q+1 \le n \leq N^2,
\end{cases}
\]
where $N_q:=cN^{2/q+1}$.
\end{proposition}
\begin{proof} For $ 1 \le n \le c_{p,q} N^2$, we use Theorem 2 of \cite{Vasil13}, the duality of Gelfand and Kolmogorov numbers, and \eqref{eq:pq1} and \eqref{eq:pq2} to obtain
\begin{align*}
\min\big\{1,n^{-1/2}N^{3/2-1/p}\big\}&=
\min\big\{1,n^{-1/2}N^{1/2+1/p^{*}}\big\}\lesssim_{p,q} d_n\big(\ell_{q^*}^N(\ell_2^N)\hookrightarrow \ell_{p^{*}}^N(\ell_2^N)\big)\\
&=c_n\big(\ell_{p}^N(\ell_2^N)\hookrightarrow \ell_{q}^N(\ell_2^N)\big)\\
&\le\|\ell_p^N(\ell_2^N)\hookrightarrow \Sc_p^N\| \cdot c_n\big(\Sc_p^N\hookrightarrow \Sc_q^N\big)\cdot \|\Sc_q^N\hookrightarrow\ell_q^N(\ell_2^N)\|\\
&\le c_n\big(\Sc_p^N\hookrightarrow \Sc_q^N\big).
\end{align*}
The case $ N^2 - N_q+1 \le n \leq N^2$ follows easily, because
\[
c_n\big(\Sc_p^N\hookrightarrow \Sc_q^N\big) \geq c_{N^2}\big(\Sc_p^N\hookrightarrow \Sc_q^N\big) = \| \Sc_q^N\hookrightarrow \Sc_p^N \|^{-1} = N^{1/q-1/p}.
\]
This completes the proof.
\end{proof}

%% % % % % % % % % % % % % % % % % % % % % % % % % % % % % % % % % % % % % % %
%\subsection{The lower bound in the case $0<p\le 1$ and $2\leq q\le \infty$}
%% % % % % % % % % % % % % % % % % % % % % % % % % % % % % % % % % % % % % % %
%
%In this final subsection we provide a lower bound in the regime $0<p\le 1$ and $2\leq q\le \infty$ for the range $1\le n\le c_{p,q}N^2$. The lower bound in the range ...
%
%As we have already explained above, a lower bound in the intermediate regime $c_{p,q}N^2 \leq n \leq N^2-cN^{2/q+1}+1$ remains open. 
%
%\begin{proposition}\label{prop: lower bound 0 < p <1 <2 <q<infty} Let $0<p\le 1$ and $2 \leq q\le \infty$. Then there is a number $c_{p,q}\in(0,1)$ such that, for all $n,N\in\N$ with $1\le n\le c_{p,q}N^2$, we have
%\[
%c_n\big(\Sc_p^N\hookrightarrow \Sc_q^N\big)\gtrsim_{p,q}
%\min\Big\{1,\Big(\frac{N}{n}\Big)^{1/p-1/q}\Big\}.
%\]
%\end{proposition}
%\begin{proof}
%...
%\end{proof}

\subsection*{Acknowledgement}
A. Hinrichs and J. Prochno are  supported  by  Project  F5513-N26 of the  Austrian  Science Fund  (FWF),  
which  is  a  part  of  the  Special 
Research  Program  ``Quasi-Monte  Carlo  Methods:  Theory  and  Applications''. J. Prochno is also supported by the Austrian Science Fund (FWF) Project P32405
``Asymptotic geometric analysis and applications''.
The research of J. Vyb\'\i ral was supported by
the grant P201/18/00580S of the Grant Agency of the Czech Republic
and by the European Regional Development Fund-Project ``Center for Advanced Applied Science'' (No. CZ.02.1.01/0.0/0.0/16\_019/0000778). We also thank Micha{\l} Strzelecki for comments on preliminary version of this paper.
Last but not least, we also gratefully acknowledge the support of the Oberwolfach Research Institute for Mathematics,
where several discussions about this problem were held during the workshop ``New Perspectives and Computational Challenges in High Dimensions'' (Workshop ID 2006b).

\bibliographystyle{plain}
\bibliography{gelfand}

\bigskip
	
	\bigskip
	
	\medskip
	
	\small
	
	\noindent \textsc{Aicke Hinrichs:} Instiute of Analysis,
	University of Linz, Altenbergerstrasse 69, 4040 Linz, Austria
	
	\noindent {\it E-mail:} \texttt{aicke.hinrichs@jku.at}
	
		\medskip
	
	\noindent \textsc{Joscha Prochno:} Institute of Mathematics and Scientific Computing,
	University of Graz, Heinrichstrasse 36, 8010 Graz, Austria
	
	\noindent 
	{\it E-mail:} \texttt{joscha.prochno@uni-graz.at}

		\medskip
		
		\noindent \textsc{Jan Vyb\'iral:} Department of Mathematics, Czech Technical University, Trojanova 13, 12000 Praha, Czech Republic
		
		\noindent 
		{\it E-mail:} \texttt{jan.vybiral@fjfi.cvut.cz}

\end{document}